\documentclass[11pt]{amsart}

\usepackage{amssymb}
\usepackage{amsmath,amsthm,geometry,enumerate}
\usepackage[dvips]{graphicx}
\usepackage{psfrag}
\usepackage{amscd}
\usepackage[all]{xy}
\normalfont

\newcommand{\RedefinitSymbole}[1]{  \expandafter\let\csname old\string#1\endcsname=#1 \let#1=\relax
\newcommand{#1}{\csname old\string#1\endcsname\,}}
\RedefinitSymbole{\forall}
\RedefinitSymbole{\exists}

\def\cf{\textit{cf.}\kern.3em}

\def\resp{\textit{resp.}\kern.3em}
\renewcommand{\k}{\kern2pt}

\numberwithin{equation}{section} 

\let\emptyset\varnothing

\DeclareMathOperator{\Aut}{Aut}

\DeclareMathOperator{\spec}{Spec}

\newtheorem{proposition}[equation]{Proposition}
\newtheorem{fact}[equation]{Fact}
\newtheorem{theorem}[equation]{Theorem}
\newtheorem{corollary}[equation]{Corollary}
\newtheorem{lemma}[equation]{Lemma}
\theoremstyle{definition}
\newtheorem{definition}[equation]{Definition}

\newtheorem{remark}[equation]{\textbf{Remark}}

\newcommand{\virg}[1]{\textquotedblleft#1\textquotedblright}
\newcommand{\Pp}[1]{\mathbb P^{#1}}
\newcommand{\C}{\mathbb C}
\newcommand{\Q}{\mathbb Q}

\newcommand{\M}[1]{\mathcal{M}_{#1}}
\newcommand{\Mm}[2]{\mathcal{M}_{#1,#2}}
\newcommand{\Mb}[1]{\overline{\mathcal{M}}_{#1}}
\newcommand{\Mmb}[2]{\overline{\mathcal{M}}_{#1,#2}}
\newcommand{\coh}[2]{H^{#1}({#2})}
\newcommand{\cohc}[2]{H_c^{#1}({#2})}

\newcommand{\pu}{\bullet}
\newcommand{\s}{\mathfrak S}
\newcommand{\Schur}[1]{\mathbb S_{#1}}

\newcommand{\Sym}{Sym}

\newcommand{\F}[2]{F({#2},{#1})}

\newcommand{\hsq}{\mathsf{HS}_\Q}
\newcommand{\chihdg}[1]{\chi^c_{\operatorname{Hdg}}(#1)}
\newcommand{\Lt}[2]{L^{#1}t^{#2}}

\begin{document}

\date{January 8, 2013}
\title{\textbf{The orbifold cohomology of moduli of genus~$3$ curves}}
\author{Nicola Pagani and Orsola Tommasi}
\address{
Nicola Pagani,
Institut f\"ur Algebraische Geometrie, Leibniz Universit\"at Hannover, Wel\-fen\-gar\-ten~1, D-30167 Hannover, Germany,
\texttt{npagani@math.uni-hannover.de}}
\address{
Orsola Tommasi,
Institut f\"ur Algebraische Geometrie, Leibniz Universit\"at Hannover, Wel\-fen\-gar\-ten~1, D-30167 Hannover, Germany,
\texttt{tommasi@math.uni-hannover.de}}

\begin{abstract}In this work we study the additive orbifold cohomology of the moduli stack of smooth genus $g$ curves. We show that this problem reduces to investigating the rational cohomology of moduli spaces of cyclic covers of curves where the genus of the covering curve is $g$. Then we work out the case of genus $g=3$. Furthermore, we determine the part of the orbifold cohomology of the Deligne--Mumford compactification of the moduli space of genus $3$ curves that comes from the Zariski closure of the inertia stack of $\M3$. 
\end{abstract}
\maketitle
\setcounter{tocdepth}{2}
\tableofcontents

\section{Introduction}

It was a remarkable discovery of the beginning of this century, anticipated in physics in the nineties, that the degree zero small quantum cohomology of a smooth Deligne--Mumford stack $X$ does not in general coincide with the ordinary cup product, and that it is, in fact, a proper ring extension of the ordinary cohomology ring of $X$. More generally, the definition of quantum cohomology and Gromov--Witten invariants for orbifolds were given in symplectic geometry by Chen and Ruan in \cite{chenruan}. The algebraic counterparts of these theories were developed by Abramovich--Graber--Vistoli in \cite{agv1}, \cite{agv2}.

The Chen--Ruan cohomology of a smooth Deligne--Mumford stack $X$ is, by definition, the degree zero part of the small quantum cohomology ring of $X$, and the orbifold cohomology of $X$ is the rationally graded vector space that underlies the Chen--Ruan cohomology algebra. The general idea, coming from stringy geometry, is that an important role in the study of $X$ is played by the so-called \emph{inertia stack} of $X$. When $X$ is a moduli space for certain geometric objects, the inertia stack of $X$ parametrizes the same geometric objects, together with the choice of an automorphism on them. 
The stack $X$ itself appears as the connected component of its inertia stack associated with the trivial automorphism, but in general there are other connected components, usually called the \emph{twisted sectors} of $X$, a terminology that originates from physics. Orbifold cohomology is simply the ordinary cohomology of the inertia stack, endowed with a different grading. Each twisted sector is assigned a rational number, called (depending on the author) \emph{degree shifting number}, \emph{age} or \emph{fermionic shift}: this number depends on the action of the given automorphism on the normal bundle to the twisted sector in $X$. Then the degree of each cohomology class of the twisted sector is shifted by twice this rational number.

In this paper, we study the inertia stack of moduli spaces $\M g$ of smooth genus $g$ curves. The starting point of our construction is that one can associate with each object $(C,\alpha)$ of the inertia stack the cover given  by quotienting $C$ by the cyclic group generated by $\alpha$. 
Following an idea of Fantechi~\cite{fantechi}, we exploit this correspondence to tackle the problem of the identification of the twisted sectors of $\M g$ by using the classical theory of cyclic (possibly ramified) covers of algebraic varieties, as developed in \cite{pardini}. We identify some discrete data in order to separate the inertia stack of $\mathcal{M}_g$ in its connected components. 
The first data are the genus of the quotient curve and the order $N$ of the automorphism; the latter is a general invariant of twisted sectors as it appears already in the definition of the inertia stack. 
Finally, the branch locus of the cover can be split in $N-1$ parts according to the local monodromy around each of its points. The last invariants are simply the degrees of each of these $N-1$ divisors. It is a recent result of Catanese \cite{catanese}
that these numerical data single out a connected component of the moduli space of connected cyclic covers. 

The twisted sectors of moduli of curves were studied with \emph{ad hoc} methods in the case of genus $2$ in \cite{spencer2}, and in the case of pointed curves of genus $1$ in \cite{pagani1}. The same approach explained in the above paragraph was used in \cite{pagani2} to identify the twisted sectors of $\mathcal{M}_{g,n}$ with $g=2$ or $n \geq 1$, in this paper we complete the picture by analyzing the more delicate case when $n=0$. A complete cohomological description of the twisted sectors of moduli of hyperelliptic curves for all genera is given in \cite{paganihyper} following a similar approach.


After having determined the connected components of the inertia stack of $\M g$ for general $g$, we study the topology of the twisted sectors in the case when $g$ equals $3$. In most cases, the cohomology of the twisted sector is computed in a rather straightforward way. The main exceptions are the twisted sectors corresponding to bielliptic and to quadrielliptic genus $3$ curves, which require a more detailed analysis. In particular, our computation of the cohomology of the moduli space of bielliptic genus $3$ curves is achieved by using a combination of Vassiliev--Gorinov's method for the computation of the cohomology of complements of discriminants with the study of certain Leray spectral sequences, following the approach of \cite{OTM4}, \cite{OTM32}. We expect that these techniques could be applied also in other cases of moduli spaces of cyclic covers, at least for small values of $g$.

  Finally, we partially extend our investigation to the orbifold cohomology of the Deligne--Mumford compactification $\overline{\mathcal{M}}_3$ of $\M3$. Specifically, we study the Zariski closure of $I(\M3)$ inside the inertia stack $I(\overline{\mathcal{M}}_3)$. 
The connected components of this compactification are precisely the connected components of $I(\Mb3)$ whose general element is a smooth curve\footnote{We can think of this situation in analogy with what happens in the theory of moduli of stable maps: it is often the case that compactifying the space of maps one introduces \virg{extraneous} components, and that the main interest is focused on the connected components of the moduli space whose general element is a map from a smooth curve.}. The study of the compactification of the inertia stack is performed using moduli stacks of admissible covers (we refer for the general theory to \cite{acv}). The cohomology of the twisted sectors contained in this compactification contributes to what we call the compactified orbifold cohomology of $\M3$.

The main results of the paper are the description of the connected components of the twisted sectors of $\mathcal{M}_g$ (Sect. \ref{sectiondefinertia}), the age for each of them (Prop.~\ref{fantechiage}) and the explicit computation of the orbifold cohomology of $\mathcal{M}_3$, which we recollect in the following two theorems: 

\begin{theorem}[Theorem~\ref{pp1m3}] The orbifold Poincar\'e polynomial of $\mathcal{M}_3$ is:
\begin{align*}
1&+t+2t^2+t^3+t^{\frac{10}{3}}+t^{\frac{7}{2}}+4t^4+2 t^{\frac{9}{2}}+ 2 t^{\frac{14}{3}}+ t^{\frac{33}{7}}+ 5 t^5+ t^{\frac{46}{9}}+ t^{\frac{36}{7}}
\\&+ 3t^{\frac{16}{3}}+ t^{\frac{38}{7}}+ 4 t^{\frac{11}{2}}
+ t^{\frac{50}{9}}
+ t^{\frac{39}{7}}
+t^{\frac{17}{3}}+ t^{\frac{40}{7}}+ t^{\frac{52}{9}}
+ t^{\frac{41}{7}}+10 t^6 +t^{\frac{43}{7}}\\
&+ t^{\frac{56}{9}}+ t^{\frac{44}{7}}+ t^{\frac{19}{3}}+ t^{\frac{45}{7}}+ t^{\frac{58}{9}}+3 t^{\frac{13}{2}}+ t^{\frac{46}{7}}+ 2 t^{\frac{20}{3}}+ t^{\frac{48}{7}}+ t^{\frac{62}{9}}+ t^{\frac{51}{7}}.
\end{align*}
\end{theorem}  
\begin{theorem}[Theorem~\ref{pp2m3}] The compactified orbifold Poincar\'e polynomial (see Definition~\ref{pp2}) of ${\mathcal{M}}_3$ is:
\begin{align*}
1& + t + 4 t^2 + 4 t^3 + t^{\frac{10}{3}} + t^{\frac{7}{2}} + 16 t^4 + 
      t^{\frac{9}{2}} + 2 t^{\frac{14}{3}} + t^{\frac{33}{7}} + 12 t^5 \\&+ 
      t^{\frac{46}{9}} + t^{\frac{36}{7}} + 5 t^{\frac{16}{3}} 
+ t^{\frac{38}{7}} + 
      5 t^{\frac{11}{2}} + t^{\frac{50}{9}} + t^{\frac{39}{7}}  +t^{\frac{40}{7}} + 
      t^{\frac{52}{9}} + t^{\frac{41}{7}}  \\&+ 31 t^6+ t^{\frac{43}{7}} + t^{\frac{56}{9}}+      t^{\frac{44}{7}} + t^{\frac{45}{7}} + t^{\frac{58}{9}} + 5 t^{\frac{13}{2}} 
+       t^{\frac{46}{7}} + 5t^{\frac{20}{3}} + t^{\frac{48}{7}} + t^{\frac{62}{9}} \\&+ 
      12 t^7 + t^{\frac{51}{7}} +2 t^{\frac{22}{3}} + t^{\frac{15}{2}} + 
      16 t^8 + t^{\frac{17}{2}} + t^{\frac{26}{3}} + 4t^9 + 4 t^{10} + 
      t^{11} + t^{12}.
\end{align*}

\end{theorem}
\noindent The ordinary additive cohomology of $\mathcal{M}_3$ and $\overline{\mathcal{M}}_3$ was studied by Looijenga and Getzler, respectively in \cite[(4.7)]{loo} and \cite[Prop.~16]{getzlertopolo}: our results are an extension of theirs.
 
 The paper is developed on three levels of generality. On the first level, we work with a general smooth Deligne--Mumford stack: we recall the theory of inertia stack, age grading, orbifold cohomology, as first developed by Chen--Ruan in \cite{chenruan} and Abramovich--Graber--Vistoli in \cite{agv1}, \cite{agv2}. The compactified orbifold cohomology is introduced for any choice of a smooth compactification $X \subset \overline{X}$. 
On the second level, we develop the above theories in the case of the moduli stack of curves $\mathcal{M}_g$ and its Deligne--Mumford compactification $\overline{\mathcal{M}}_g$. Finally the third level is devoted to work out all the details in the case $g=3$.

We work with cohomology, as in the seminal paper of Chen--Ruan. Nevertheless, many of the techniques we use are algebraic, thus closer to the Abramovich--Graber--Vistoli's approach. We work over the field of complex numbers, and cohomology is always taken with rational coefficients.

\label{intro}

\subsection{Notation}\label{notation}

 By a stack, we shall always mean a Deligne--Mumford stack, of finite type over $\mathbb{C}$. In this context, the canonical map from a stack to its coarse moduli space induces an isomorphism in cohomology: we will often identify the two cohomologies by means of this isomorphism. We adopt the convention that orbifold cohomology is the graded vector space underlying Chen--Ruan cohomology, where the latter carries the additional ring structure.
 
 If $X$ is an algebraic variety, or, more generally, a Deligne--Mumford stack, and $G$ is a finite group acting on it, we denote the quotient Deligne--Mumford stack by $[X/G]$. When $X$ is a point, we write $B G$ for $[\spec{\mathbb{C}}/G]$: the moduli stack that classifies principal $G$-bundles.

In our work, we shall consider the cohomology with its mixed Hodge structures. We shall denote by $\Q(-k)$ the Hodge structure of Tate of weight $2 k$. The class of $\Q(-1)$ in the Grothendieck group $K_0(\hsq)$ of rational Hodge structures will be denoted by $L=[\Q(-1)]$.
 
Results on the cohomology with compact support of a quasi-projective variety (or stack with quasi-projective coarse moduli space) $X$ shall often be expressed by means of its Euler characteristic in $K_0(\hsq)$. Following \cite[Definition~5.5.2]{PS}, we call this Euler characteristic the \emph{Hodge--Grothendieck character for compact support} of $X$ and denote it by
$$
\chihdg X = \sum_{i\in\mathbb N}(-1)^i[\cohc iX]\in K_0(\hsq).
$$
Hodge--Grothendieck characters for compact support are sometimes called Serre characteristics in the literature.

Similarly, to state results on cohomology with compact support in a concise way, we shall express them as polynomials with coefficients in the Grothendieck group of Hodge structures:
$$
P_X(t) = \sum_{i\in\mathbb N}[\cohc iX]t^i\in K_0(\hsq)[t].
$$ 

We work with cyclic covers $f:\;C \to C'$, where $C$ and $C'$ are, respectively, the covering and the covered space. If a cyclic cover is not \'etale, the source and the target contain, respectively, ramification and branch points.

We shall denote the symmetric group in $d$ letters by $\s_d$ and the group of $k$th roots of unity by $\mu_k$.

\ \\
\section{The inertia stacks}

\subsection{Definition of the inertia stack}
\label{sectiondefinertia}
In this section we recollect some basic notions concerning the inertia stack. For a more detailed study, we address the reader to \cite[Section 3]{agv2}. 

We introduce the following natural stack associated to a stack $X$, which points to where $X$ fails to be an algebraic space.

\begin{definition} \label{definertia} (\cite[4.4]{agv1}, \cite[Definition 3.1.1]{agv2}) Let $X$ be a stack. The \emph{inertia stack} $I(X)$ of $X$ is defined as:
$$
I(X) := \coprod_{
\begin{smallmatrix}
N \in \mathbb{N}_{>0}\\
\end{smallmatrix}
} I_N(X)
$$
where $I_N(X)(S)$ is the following groupoid:
\begin{enumerate}
\item The objects are pairs $(\xi, \alpha)$, where $\xi$ is an object of $X$ over $S$, and $\alpha: \mu_N \to \Aut(\xi)$ is an injective homomorphism;

\item The morphisms are the morphisms $g: \xi \to \xi'$ of the groupoid $X(S)$, satisfying $g \cdot \alpha(1)= \alpha'(1) \cdot g$.
\end{enumerate}
The inertia stack comes with a natural forgetful map $f:I(X) \to X$.

We also define $I_{TW}(X):= \coprod_{N>1}I_N(X)$.
The connected components of $I_{TW}(X)$ are called \emph{twisted sectors} of the inertia stack of $X$, or simply twisted sectors of $X$. 
\end {definition}

We remark that, by its very definition, $I_N(X)$ is an open and closed substack of $I(X)$, but it rarely happens  that it is connected. One special case is when $N$ equals $1$: in this case the map $f$ restricted to $I_1(X)$ induces an isomorphism of the latter with $X$. The connected component $I_1(X)$ will be referred to as the \emph{untwisted sector}.
 We also observe that after the choice of a generator of $\mu_N$, we obtain an isomorphism of $I(X)$ with $I'(X)$, where the latter is defined as the ($2$-)fiber product $X \times_{X \times X} X$ where both morphisms $X \rightarrow X \times X$ are the diagonals.

\begin{remark} \label{mappaiota} There is an involution $\iota: I_N(X) \to I_N(X)$, which is induced by the map $\iota': \mu_N \to \mu_N$ given by  $\iota'(\zeta):= \zeta^{-1}$.
\end{remark}
\noindent The inertia stack, which we have just defined, is the fundamental ingredient in the definition of orbifold cohomology (Chen--Ruan cohomology as a vector space). We observe that, at this level, we do not need $X$ to be smooth nor proper.
\begin{definition} (\cite{chenruan}) \label{defcoomorb1} Let $X$ be a stack. The \emph{orbifold cohomology} (with rational coefficients) of $X$ is defined as a vector space as:
$$
H^\pu_{CR}(X):= \coh\pu{I(X)}.$$
\end{definition}

Now if $X \hookrightarrow \overline{X}$ is an open dense embedding, we can define an intermediate space between $I(X)$ and $I(\overline{X})$, namely:

\begin{definition} \label{ibar} Given a compactification of a stack $i:\;X \to \overline{X}$, we define the \emph{compactified inertia stack} of $X$ as the stack 
$$\overline{I}(X)= \coprod_{\begin{smallmatrix}N \in \mathbb{N}\end{smallmatrix}} \overline{I}_N(X)$$ 
where $\overline{I}_N(X)$ is the stack of all connected components $Y$ of $I_N(\overline{X})$ such that $i^*Y \neq \emptyset$. We can thus define the \emph{compactified orbifold cohomology} as the following vector space:
$$
\overline{H}^{\pu}(X):= \coh\pu{\overline{I}(X)}.$$
\end{definition}
In the following sections, we will study the inertia stack for moduli of smooth genus $g$ curves and its compactified inertia stack with respect to the Deligne--Mumford compactification. In the first case we will use the theory of cyclic covers of smooth curves (see \cite{pardini}), in the second we will use the theory of admissible covers as developed in \cite{acv}.

\begin{remark} \label{stackyremark} The orbifold cohomology of $X$ only depends upon the topological space (coarse moduli space) underlying $I(X)$.  In \cite{agv2}, the authors introduce two notions related to the inertia stack: the \emph{stack of cyclotomic gerbes} (\cite[Definition 3.3.6]{agv2}) and the \emph{rigidified inertia stack} (\cite[3.4]{agv2}), showing in \cite[3.4.1]{agv2} that they are equivalent categories. It is relevant to observe that all these different notions of inertia stacks share the same coarse moduli space, and therefore they give rise to the same orbifold cohomology theory. 
\end{remark}

\subsection{The inertia stack of moduli of genus $g$ smooth curves}
\label{inertiamg}

We want to study the twisted sectors of the inertia stack of moduli of smooth genus $g$ curves. 
For this, we study the moduli stacks of cyclic ramified covers of curves of genus $g'<g$. This approach is due to Fantechi \cite{fantechi}, and builds on the theory of abelian covers of algebraic varieties (see Pardini \cite{pardini}). The first author has used this approach in \cite[Section 2.b]{pagani2} for the simpler cases when the genus is $2$, or when there is a positive number of marked points. A description of the theory of abelian covers in the case of curves and of cyclic groups that is closely related to the one we use can be found in \cite[Sections~1 and~2]{catanese}. A similar construction for covers of prime order was studied in \cite{cornalba}.

We start by summarizing informally the description of cyclic covers we will use in this paper. 
\begin{fact} \label{pardiniprop} (\cite[Proposition 2.1]{pardini}) Let $C'$ be a smooth genus $g'$ curve. Then the following data are equivalent:
\begin{itemize}
\item A cyclic (possibly ramified) $\mu_N$-cover $\psi:C \to C'$, where $C$ is a smooth curve, possibly disconnected;
\item A sequence of $N-1$ smooth effective divisors $D_1, \ldots, D_{N-1}$ (with pairwise disjoint support, possibly empty), a line bundle $L$ on $C'$ together with an isomorphism $\phi:L^{\otimes N} \to \mathcal{O}_{C'}(\sum_i i D_i)$.
\end{itemize}
\end{fact}

With this result in mind, let us define:
\begin{definition} \label{admissible} Let $g>1$ be an integer. A $g$-admissible datum is an $(N+1)$-tuple of nonnegative integers $A=(g',N;d_1,\ldots,d_{N-1})$ with $N \geq 2$ and $g'\leq g$, satisfying the following conditions:
\begin{itemize}
\item Riemann--Hurwitz formula 
\begin{equation}\label{riemann-hurwitz}
2g-2=N (2 g'-2)+  \left(\sum d_i \ \gcd(i,N)\left(\frac{N}{\gcd(i,N)}-1\right)  \right);
\end{equation}
\item the structural equation of abelian covers 
\begin{equation}\label{abcovers}
\sum i \ d_i= 0 \mod N.
\end{equation}
\end{itemize}
The integers $N$ and $g'$ will be called respectively the \emph{order} and the \emph{base genus} of the $g$-admissible datum $A$.
\end{definition}

Note that, for a fixed $g$, the set of all $g$-admissible data $A$ is finite. With every admissible $g$-datum, we associate the integers $d= \sum d_i$, and a disjoint union decomposition $\{1,\ldots, d\}= \coprod_{i=1}^{N-1} J_i$ by:
$$
J_i:= \left\{j | \ \sum_{l<i} d_l < j < \sum_{l \leq i} d_l \right\}.
$$
Moreover, we will denote by $S_A$ the subgroup of $\s_d$ (the symmetric group on $d$ elements) defined by  
$ S_A := \left\{ \sigma| \ \sigma (J_i)=J_i \right\}$. 

 We shall now construct the twisted sectors of $\mathcal{M}_g$ as stacks of cyclic $N$-covers of curves of genus $g'$, with branch locus of type $d_1, \ldots, d_{N-1}$.

\begin{definition} \label{settoretwistato} Let $A$ be a $g$-admissible datum. We define the stack $\mathcal{M}_A$ whose objects over a scheme $S$ are $(N+2)$-tuples $(C,D_1, \ldots, D_{N-1}, L, \phi)$, where $C$ is a smooth family of genus $g'$ curves over $S$, the $D_i$ are sections of $(\Sym^{d_i}_SC\setminus \Delta_{d_i})\rightarrow S$ (where $\Delta_{d_i}$ denotes the big diagonal) defining disjoint divisors on $C$, $L$ is a line bundle and $\phi:L^{\otimes N} \to \mathcal{O}_{C} (\sum  i D_i)$ is an isomorphism.

The stack $\mathcal{M}_A'$ is defined as the open and closed substack of $\mathcal{M}_A$ whose objects under the correspondence described in Fact \ref{pardiniprop} are connected covers.
\end{definition}
\begin{remark} \label{condizioneconnessione} Let $A$ be a $g$-admissible datum, and define $k$ as the greatest common divisor of $N$ and all the $i$ with $d_i \neq 0$.
The condition of connectedness for the cover in the last sentence of Definition~\ref{settoretwistato} is satisfied when
\begin{equation}\label{orderm}
L^{\otimes\frac{N}{k}}~\otimes~\mathcal{O}_C(-\sum\frac{i}{k}D_i)
\end{equation}
has precisely order $k$ in the Picard group of $C$. When $g'$ equals $0$, then $\mathcal{M}_A$ is always connected, thus $\mathcal{M}_A'$ is empty when $k>1$, or it coincides with $\mathcal{M}_A$ when $k=1$. When $g'>0$, fixing a proper divisor of $k$ as the order of \eqref{orderm} determines an open and closed substack of $\mathcal{M}_A$. As a consequence of Theorem \ref{connessione}, we see \emph{a posteriori} that each of these open and closed substacks (in particular also $\mathcal{M}_A'$) is in fact connected.
\end{remark} 

\begin{remark} \label{esisteolofisso} Let us be more explicit about the morphisms of $\mathcal{M}_A(S)$. Let $(C, D_i, L,$ $\phi)$ and $(C', D_i', L', \phi')$ be two objects as in Definition~\ref{settoretwistato}. Then a morphism between them is a couple of isomorphisms $(\sigma:\;C\rightarrow C', \tau: \sigma^*L'\rightarrow L)$ satisfying the following conditions: The map $\sigma$ is an isomorphism of curves such that $\sigma^*(D_i')=D_i$ and $\tau$ is an isomorphism of line bundles that makes the following diagram commute:
\begin{equation} \label{diagrammone}
\xymatrix{ \sigma^*(L'^{\otimes N}) \ar[r]^{\hspace{-0.5cm}\sigma^*(\phi')} \ar[d]^{\tau^{\otimes N}} & \sigma^*\left( \mathcal{O}_{C'}( \sum i D_i')\right) \ar[d]^{\gamma} \\
L^{\otimes N} \ar[r]^{\hspace{-0.5cm}\phi} & \mathcal{O}_C(\sum i D_i),}
\end{equation}
where we denoted by $\gamma$ the isomorphism induced by $\sigma$. 
A different definition for a morphism between the two families $(C, D_i, L, \phi)$ and $(C', D_i', L', \phi')$ is the following: a single isomorphism $\sigma: C \to C'$ satisfying $\sigma^*(D_i')=D_i$ and such that there exists an isomorphism $\tau:\;\sigma^*L'\rightarrow L$ making diagram~\eqref{diagrammone} commute. These two different definitions of morphisms give rise to two different stacks, which share the same coarse moduli space. \end{remark}

\begin{remark} Let us denote by $\mathcal{M}_{g',d}(B \mu_N)$ the open substack of the moduli stack of stable maps $\mathcal{K}_{g',d}(B \mu_N)$ whose source curve is smooth. The moduli stack $\mathcal{K}_{g',d}(B \mu_N)$ is defined in \cite{abvis2}; see also \cite{acv}, where $\mathcal{K}_{g',d}(B \mu_N)$ is denoted $\mathcal{B}_{g',d}(\mu_N)$. We observe that the stack $\mathcal{M}_A$ we have just defined (with the first definition of morphisms in Remark \ref{esisteolofisso}) is an open and closed substack of the quotient stack $[\mathcal{M}_{g',d}(B \mu_N)/S_A]$ prescribed by the assignment of the ramifications $d_1, \ldots, d_{N-1}$.
\end{remark}

We now show that the moduli stacks $\mathcal M_A'$ we have just constructed constitute open and closed substacks of the inertia stack of $\mathcal{M}_g$.

\begin{corollary} \label{corrispondenza} Let us fix $g,N >1$. Then the stack $I_N(\mathcal{M}_g)$ of Definition \ref{definertia} is isomorphic to the disjoint union of all nonempty stacks $\mathcal{M}_A'$ 
for all $g$-admissible data $A=(g',m,d_1,\ldots,d_{m-1})$ with order $m$ equal to $N$.
\end{corollary}
\begin{proof}
Follows from Definitions \ref{admissible}, \ref{settoretwistato} and by adapting the proof of~\cite[Theorem~2.1, Proposition~2.1]{pardini} to this relative case (cf. Fact \ref{pardiniprop}). Indeed, there is a base-preserving equivalence of categories: \begin{equation} \label{correspondence} I_N(\mathcal{M}_g)(S) \to \coprod_A \mathcal{M}'_A(S),\end{equation}
where in the right hand side the disjoint union is taken over all $g$-admissible data with order $N$. We sketch the proof of this well--known fact, by explicitly defining the (functorial) correspondence~\eqref{correspondence}.
Let us assume that a  $\mu_N$-cover $\psi: X \to C$ is given (over a base $S$).
Over each point $s\in S$, the branch divisor $D_s$ of the cover $X_s\rightarrow C_s$ can be split according to local monodromy in smooth effective  divisors $D_{1,s}, \ldots, D_{N-1,s}$, having pairwise disjoint support. Identifying these divisors with sections of the appropriate symmetric product of $C\rightarrow S$ gives the $D_i$. Each $D_i$ defines a codimension $1$ subscheme of $C$ which does not contain any fibre of $C\rightarrow S$, hence they give rise to effective Cartier divisors on $C$.
At this point, we only need to construct the line bundle $L$ together with the isomorphism $\phi:\;L^{\otimes N} \to \mathcal{O}_{C}(\sum_i i D_i)$. 
Since the cover is nontrivial, the action of $\mu_N$ on the push-forward sheaf $\psi_*(\mathcal{O}_X)$ defines a splitting as a direct sum of line bundles:
$$
\psi_* \mathcal{O}_X= L_0 \oplus \ldots \oplus L_{N-1}
$$
where $L_i$ is the subsheaf of $\psi_*(\mathcal{O}_X)$ of sections where $\mu_N$ acts with weight $i$. The line bundle $L$  is then defined as $L_1^{\vee}$. By viewing the sections of $L=L_1^{\vee}$ as functions on the total space of the line bundle $L_1$, and hence also as functions on its trivial section $C$, we obtain an identification
$$
\phi:\; L_1^{\otimes N} \otimes \mathcal{O}_C (\sum_i i D_i ) \to \mathcal{O}_C.
$$

The correspondence just defined is essentially surjective. Indeed, if the line bundle $L$ is given over $C'$, we set $L_1:= L^{\vee}$. The line bundles $L_a$ can then be defined as:
\begin{equation} \label{ellea}
L_a:= L_1^{\otimes a} \otimes \bigotimes_{i=1}^{N-1} \mathcal{O}_C(D_i)^{- [\frac{a i}{N} ]} \quad a=0,\ldots, N-1
\end{equation}
Now the number $[\frac{a i }{N}]+ [\frac{b i }{N}]- [\frac{(a+b) i }{N}]$ can either be $0$ or $1$. In both cases, the canonical sections of the line bundles:
$$
\mathcal{O}_C(D_i)^{[\frac{a i }{N}]+ [\frac{b i }{N}]- [\frac{(a+b) i }{N}]}
$$
permit the definition of a ring structure over $R:=\bigoplus_{i=0}^{N-1} L_i$. Now the normalization of the spectrum of $R$ reconstructs the smooth $\mu_N$-cover of $C'$ up to isomorphism of $\mu_N$-cover. 
If one considers the definition of morphisms in the groupoid $\mathcal{M}_{A}'(S)$ given in Remark~\ref{esisteolofisso}, 
one can also check that the correspondence is fully faithful, hence an equivalence of categories.  
\end{proof}

The moduli spaces $\mathcal{M}_A'$ are indeed connected. This would follow, after some work, from \cite[Theorem 1.2]{edmonds}. We refer to a more recent work of Catanese, where the result we need is stated in the same language used in this paper.

\begin{theorem} (\cite[Theorem 2.4]{catanese}) \label{connessione} Let $A$ be a $g$-admissible datum. When the stack $\mathcal{M}_A'$ is nonempty, it is connected.
\end{theorem}

In particular, as a consequence of this connectedness result, the nonempty moduli stacks $\mathcal{M}_A'$ give all the twisted sectors of the inertia stack of $\mathcal{M}_g$. 

\begin{remark}  Working with the second definition in Remark \ref{esisteolofisso} of morphism for the stack $\mathcal{M}_A$, one obtains a decomposition of the rigidified inertia stack (see Remark~\ref{stackyremark}) of $\mathcal{M}_g$.
\end{remark}

\section{The inertia stack of $\mathcal{M}_3$} \label{inertiam3}
In this section, we study the geometry of the twisted sectors of the inertia stack of the moduli space of smooth, genus $3$ curves. We determine the cohomology of all these twisted sectors as a graded vector space with Hodge structures. We shall state these results in the form of polynomials with coefficients in $K_0(\hsq)$ (see Section~\ref{notation}).

Our approach is based on the correspondence between twisted sectors and $g$-admissible data introduced in the previous section. Of course, in the case of genus $3$ also a direct approach is possible by classifying all automorphisms of plane quartic curves (as in e.g. \cite[Lemma~6.5.1]{dolgachev-topics}) and of all hyperelliptic genus $3$ curves. However, our approach seems more suitable for cohomological computations and has the advantage that it generalizes to higher genus.

If $X$ is a twisted sector of $I(\mathcal{M}_3)$, we have seen in the previous section that $X \cong \mathcal{M}_A$ for $A$ a certain $3$-admissible datum. We start by considering the admissible data with $g'=0$.

\begin{proposition}\label{basegenus0}
There are $43$ different $3$-admissible data $A$ with $g'=0$ that parametrize connected covers. The complete list of these admissible data and of the Hodge--Grothendieck characters for compact support of the associated twisted sectors $\mathcal M_A$ is given in Tables~\ref{base0dim>0} and~\ref{base0dim0}.
\end{proposition}

\begin{table}
\begin{center}
\begin{tabular}{|c|c|c|c|} \hline 
&&&\\[-10pt]
$A$ & $\chihdg{\mathcal M_A}$ & $\chihdg{\overline{\mathcal M}_A}$&$a(\mathcal M_A)$\\[1pt]
\hline \hline
&&&\\[-9pt]
$(2;8)$&$L^5$&$L^5+3L^4+6L^3+6L^2+3L+1$&$\frac{1}{2}$\\[1pt]
$(3;4,1)$&$L^2$&$L^2+2L+1$&$\frac{5}{3}$\\[1pt]
$(3;1,4)$&$L^2$&$L^2+2L+1$&$\frac{7}{3}$\\[1pt]
$(4;4,0,0)$& $L$&$L+1$&$2$\\[1pt]
$(4;0,0,4)$ & $L$&$L+1$&$3$\\[1pt] 
$(4;2,3,0)$ &$L^2-L$ &$L^2+2L+1$&$\frac{7}{4}$\\[1pt]
$(4;0,3,2)$ &$L^2-L$&$L^2+2L+1$&$\frac{9}{4}$\\[1pt] 
$(4;2,0,2)$ &$L-1$&$L+1$&$\frac{5}{2}$\\[1pt]
$(6;1,0,2,0,1)$ &$L-1$ &$L+1$&$\frac{5}{2}$\\[1pt] 
$(6;1,0,1,2,0)$&$L-1$ &$L+1$&$\frac{8}{3}$\\[1pt] 
$(6;0,2,1,0,1)$&$L-1$ &$L+1$&$\frac{7}{3}$\\[1pt] 
\hline
\end{tabular}\end{center}
\caption{\label{base0dim>0}Positive-dimensional twisted sectors. For the sake of brevity we omit $g'=0$ from the notation of the admissible datum.}
\end{table}

\begin{table}
\begin{center}
\begin{tabular}{|c|c||c|c|} \hline 
$3$-admissible with $g'=0$  &Age&$3$-admissible with $g'=0$  &Age\\
\hline \hline
&&&\\[-9pt]
$(7;2,0,0,0,1,0)$ &$\frac{20}{7}$ & $(9;1,0,1,0,1,0,0,0)$ &$\frac{26}{9}$\\[1pt] 
$(7;0,1,0,0,0,2)$ &$\frac{22}{7}$ & $(9;0,0,0,1,0,1,0,1)$ &$\frac{28}{9}$\\[1pt]
$(7;1,1,0,1,0,0)$ &$3$ & $(9;0,1,1,1,0,0,0,0)$ &$\frac{29}{9}$\\[1pt] 
$(7;0,0,1,0,1,1)$ &$3$ & $(9;0,0,0,0,1,1,1,0)$ &$\frac{25}{9}$\\[1pt]  
$(7;1,0,2,0,0,0)$ &$\frac{23}{7}$ & $(12;10100001000)$&$\frac{10}{3}$\\[1pt] 
$(7;0,0,0,2,0,1)$ &$\frac{19}{7}$ & $(12;00010000101)$ &$\frac{8}{3}$\\[1pt] 
$(7;0,2,1,0,0,0)$ &$\frac{24}{7}$ & $(12;10001100000)$ &$\frac{13}{4}$\\[1pt] 
$(7;0,0,0,1,2,0)$ &$\frac{18}{7}$ & $(12;00000110001)$ &$\frac{11}{4}$\\[1pt]  
$(8;2,0,0,0,0,1,0)$ &$\frac{13}{4}$ & $(12;00111000000)$ &$\frac{8}{3}$\\[1pt] 
$(8;0,1,0,0,0,0,2)$ &$\frac{11}{4}$ & $(12;00000011100)$ &$\frac{10}{3}$\\[1pt]
$(8;1,1,0,0,1,0,0)$ &$3$ & $(14;1000011000000)$ &$\frac{51}{14}$\\[1pt] 
$(8;0,0,1,0,0,1,1)$ &$3$ & $(14;0000001100001)$ &$\frac{33}{14}$\\[1pt]
$(8;0,1,2,0,0,0,0)$&$\frac{11}{4}$ & $(14;0100101000000)$ &$\frac{41}{14}$\\[1pt] 
$(8;0,0,0,0,2,1,0)$ &$\frac{13}{4}$ & $(14;0000001010010)$ &$\frac{43}{14}$\\[1pt] 
$(9;1,1,0,0,0,1,0,0)$ &$\frac{31}{9}$ & $(14;0011001000000)$ &$\frac{45}{14}$\\[1pt] 
$(9;0,0,1,0,0,0,1,1)$ &$\frac{23}{9}$ & $(14;0000001001100)$&$\frac{39}{14}$\\[2pt] 
\hline
\end{tabular}
\end{center}
\caption{\label{base0dim0}$0$-dimensional twisted sectors. For the sake of brevity we omit $g'=0$ from the notation of the admissible datum.}
\end{table}

\begin{proof}
If $A$ is a $g$-admissible datum with $g'=0$, then it is easy to see that $\mathcal{M}_A \cong [\mathcal{M}_{0,d}/S_A]$ holds. Therefore one has $H_c^\pu(\mathcal{M}_A) \cong H_c^\pu(\mathcal{M}_{0,d})^{S_A}$, the $S_A$-invariant part of the cohomology with compact support of $\Mm0d$. 
Hence, for every $A$, the cohomology with compact support of $\mathcal M_A$ can be computed from the description of the cohomology of $\Mm0d$ as a representation of the symmetric group $\s_d$, which is known for every $d\geq3$ by work of Getzler \cite[5.6]{getzleroperads} (see also \cite[Theorem 2.9]{kisinlehrer}).

In our case, we need to work with connected covers, i.e. we restrict to $3$-admissible data that satisfy the connectedness condition explained in Remark \ref{condizioneconnessione} for $g'=0$.
We obtain their list (which we give in Table~\ref{base0dim>0} and~\ref{base0dim0}) by finding all solutions of equations~\eqref{riemann-hurwitz}, \eqref{abcovers} for $g=3$, and the condition explained in Remark \ref{condizioneconnessione} for $g'=0$. 
Using Getzler's formulas we compute their Hodge--Grothendieck characters for compact support, i.e. the Euler characteristic of their cohomology with compact support in the Grothendieck group of Hodge structures.
\end{proof}

\begin{remark}
If $A$ is an admissible datum with $g'=0$, the Hodge--Grothendieck character for compact support of $\mathcal M_A$  determines uniquely the cohomology of $\mathcal M_A$, because its $k$-th compactly supported cohomology carries a pure Hodge structure of weight $2\dim(\mathcal M_A)-k$.
This property holds for the space $\Mm 0d$ and follows from  the structure of $\Mm 0d$ as a complement of hyperplanes in $\C^{d-3}$. Since $\cohc\pu{\mathcal M_A}$ is a subring of $\cohc\pu{\Mm 0d}$, it holds for $\mathcal M_A$ as well.
\end{remark}

It is easy to see that there are exactly four $3$-admissible data with $g'>0$.
Following Definition \ref{settoretwistato}, they correspond to the four moduli stacks of cyclic covers:
\begin{equation}\label{abcd}
\mathcal M_{(1,2;4)},\ \ 
\mathcal M_{(1,3;1,1)},\ \ 
\mathcal M_{(1,4;0,2,0)}\ \text{ and }
\mathcal M_{(2,2;0)}.
\end{equation}
In view of Definition \ref{settoretwistato}, the moduli stacks $\mathcal M_{(g',N,d_1,\dots,d_{N-1})}$ parametrize objects of type $\left(C,L,D_1,\ldots,D_{N-1},\phi\right)$, where $C$ is a curve of genus $g'$, the $D_i$ are disjoint effective divisors of prescribed degrees $d_i$ and  $\phi:\;L^{\otimes N}\rightarrow\mathcal O_C(\sum_i iD_i)$ is an isomorphism. Hence, it suffices to compute the cohomology of the following four stacks:
\begin{align*}
\mathcal{A}&= \left\{(C,D_1,L) |\ g(C)=1 ,\ \deg(D_1)=4, \ L^{\otimes2} \cong \mathcal O_C(D_1) \right\},
\\
\mathcal{B}&= \left\{(C,D_1,D_2,L) \left|\begin{array}{l}g(C)=1, \deg(D_1)=\deg(D_2)= 1,\\ L^{\otimes 3} \cong \mathcal O_C(D_1+2 D_2)\end{array}\right.\right\},
\\
\mathcal{C}&= \left\{(C,D_2,L) |\ g(C)=1, \ \deg(D_2)=2,  \ L^{\otimes 4} \cong \mathcal O_C(2 D_2) \right\},
\\
\mathcal{D}&= \left\{(C,D_1,L) |\ g(C)=2,  \ L^{\otimes2} \cong \mathcal O_C \right\}.
\end{align*}

It is easy to see that $\mathcal{A}$ and $\mathcal{B}$ are connected, while $\mathcal{C}$ and $\mathcal{D}$ are not. The stack $\mathcal{C}$ has two open and closed substacks $\mathcal{C}'$ and $\mathcal{C}''$ that correspond, respectively, to the two conditions $L^{\otimes 2}\not\cong \mathcal O(D_2)$ and $L^{\otimes 2} \cong \mathcal O(D_2)$. The stack $\mathcal{D}$ has two open and closed substacks $\mathcal{D}'$ and $\mathcal{D}''$ that correspond, respectively, to the conditions $L\not\cong\mathcal O_C$ and $L \cong \mathcal O_C$. Observe that $\mathcal{C}''$ and $\mathcal{D}''$ parametrize \emph{disconnected} covers.

In the remainder of this section, we study the cohomology  of $\mathcal{B}$ and $\mathcal{D}'$, while we postpone the analogous computation for $\mathcal{A}$ and $\mathcal{C}'$ to the following section.

In the following lemma, we let $X_1(3)$ be the closed substack in $\mathcal{M}_{1,2}$ of curves  $(C, p_1, p_2)$ such that $p_2$ is a point of $3$-torsion for the elliptic curve $(C,p_1)$.

\begin{lemma} The coarse moduli space of $\mathcal{B}$ is isomorphic to the coarse moduli space of $\mathcal{M}_{1,2} \setminus X_1(3)$.
\end{lemma}

\begin{proof}
The moduli stack $\mathcal{B}$ parametrizes curves $C$ of genus $g'=1$, two distinct points $x,y\in C$ and a line bundle $L$ of degree $1$ on $C$. Let $\mathcal{C}_{1,2}$ be the universal curve over $\mathcal{M}_{1,2}$: it parametrizes genus $1$ curves $C$ with three points $x_1, x_2, q$ such that $x_1 \neq x_2$. We define $\tilde{\mathcal{B}}$ to be the irreducible codimension $1$ substack of $\mathcal{C}_{1,2}$ defined by the following constraint on the three points: 
$$
\tilde{\mathcal{B}}= \{(C, x_1, x_2, q)| \ 3q \equiv x_1 + 2 x_2 \}.
$$
Now it is clear that $\mathcal{B}$ and $\tilde{\mathcal{B}}$ share the same coarse moduli space. This can be checked by associating to a triple $(C,x,y,L)$ the triple $(C, x_1, x_2, q)$ where $q$ is the point on the curve $C$ determined by the isomorphism class of the line bundle $L$.
The diagram:
$$
\xymatrix@C=4pt{
{\tilde{\mathcal B} = \{3q=x_1+2 x_2\}}\ar[d]\ar@{^{(}->}[rr]&&{\mathcal C_{1,2}}\ar[d]^{\pi_x}\ar@{}[r]|{\ni}&{(x,y,q)}\ar@{|->}[d]\\
{\Mm12}\ar@{^{(}->}[rr]&&{\Mmb 12}\ar@{}[r]|{\ni}&{(y,q)}
}
$$
is cartesian. Indeed, the restriction of $\pi_x$ to the closed locus $\tilde{\mathcal B} \subset \mathcal{C}_{1,2}$ takes values in $\mathcal{M}_{1,2}$, since $q=y, 3q=x+2y \implies x=y$. Furthermore, the restriction of $\pi_x$ to $\tilde{\mathcal B}$ is an isomorphism onto the image locus, \textit{i.e.} the points where $3q \neq 3y$. From this the claim follows.
\end{proof}

\noindent From this description, we can deduce the cohomology with compact support of $\mathcal B$.

\begin{corollary} \label{psb} The cohomology with compact support of  $\mathcal{B}$ is given by
$$
P_{\mathcal{B}} (t):= \sum_{i\in\mathbb N}[\cohc i{\mathcal B}]t^i= L^2 t^4 +L t^3 + t^2.
$$
\end{corollary}
\begin{proof}
The cohomology with compact support of $\mathcal{M}_{1,2}$ is concentrated in degree $4$, while the cohomology with compact support of $X_1(3)$ can be deduced from the fact that its coarse moduli space is a $\mathbb{P}^1$ minus two points.
This classical result can be proved directly by considering the coarse moduli space of $X_1(3)$ as the quotient of the pointed rational curve $\Pp1\setminus\{[0,1],[1,-3\zeta_3^k]\}$ parametrizing the Hesse pencil $\lambda(x_0^3+x_1^3+x_3^3)+\mu x_0x_1x_2=0$ by the action of $\mu_3$ generated by $[\lambda,\mu]\mapsto[\lambda,\zeta_3\mu]$.
Then the result follows from the long exact sequence of compactly supported cohomology, associated to the inclusion of $X_{1}(3)$ in $\mathcal{M}_{1,2}$ with complement isomorphic to~$\mathcal B$. 
\end{proof}

\noindent Now we study the stack $\mathcal{D}'$. It turns out that we can describe it as a quotient of a moduli stack of genus $0$ curves with marked points, by the action of a subgroup of the symmetric group that symmetrizes some of the points. This enables us to compute its cohomology from Getzler's formulas (\cite{getzleroperads}).

\begin{lemma} \label{lemmad'} The coarse moduli space of $\mathcal{D}'$ is $\mathcal{M}_{0,6}/\s_4 \times \s_2$.
\end{lemma}
\begin{proof}
The claim follows from the well known fact that every nontrivial square root of the structure sheaf on a smooth genus $2$ curve $C$ is of the form $\mathcal O(x_1-x_2)$ where $x_1$ and $x_2$ are distinct Weierstrass points of $C$, and that this expression is unique up to changing the order of $x_1$ and $x_2$.
\end{proof}
\begin{corollary} \label{psd'} The cohomology with compact support of $\mathcal{D}'$ is given by
$$
P_{\mathcal{D}} (t)= L^3 t^6+ L^2 t^5.
$$
\end{corollary}

\subsection{The geometry of the twisted sectors $\mathcal{A}$ and $\mathcal{C}'$}

The aim of this section is to prove the following two results:
\begin{proposition} \label{psa} The cohomology with compact support of $\mathcal{A}$ is expressed by 
$$
P_{\mathcal{A}} (t)= L^4 t^8+ L^2 t^5. 
$$
\end{proposition}

\begin{proposition} \label{psc'} The cohomology with compact support of $\mathcal{C}'$ is given by
$$
P_{\mathcal{C}'} (t)= L^2 t^4+  L t^3+t^2. $$
\end{proposition}

This completes our cohomological analysis of the inertia stack of $\M3$. In particular, 
we are now able to produce the dimension of the orbifold cohomology as a vector space $H^\pu_{CR}(\mathcal{M}_3)$ from the Propositions~\ref{basegenus0}, \ref{psa}, \ref{psc'} and the Corollaries~\ref{psb} and~\ref{psd'}.

\begin{corollary} The orbifold cohomology of $\mathcal{M}_3$ has dimension $62$.
\end{corollary}

\noindent In Section \ref{crpp3} we shall describe the $\Q$-graded structure of this vector space in Theorem~\ref{pp1m3}.

\subsubsection{The cohomology of the moduli space $\mathcal{A}$.}
Recall that the moduli space $\mathcal A$ parametrizes bielliptic genus $3$ curves. We described it as the moduli stack of genus~$1$ curves with a set $\{x_1,x_2,x_3,x_4\}$ of (unordered) marked points and a line bundle $L$ such that $L^{\otimes 2}=\mathcal O(x_1+x_2+x_3+x_4)$. 

The divisor $D_1:=x_1+x_2+x_3+x_4$ defines an embedding of $C$ in $\Pp3$; the image is the complete intersection of two quadrics. If we consider $C$ as a curve in $\Pp3$, the square roots of $\mathcal O(D_1)$ correspond to divisors cut by planes in $\Pp3$ that are tangent to $C$ at two points (possibly coinciding) with multiplicity $2$. For a fixed $C$ this gives four distinct line bundles. They can be constructed explicitly as the $g^1_2$ cut by the ruling of each of the four singular quadrics in $\Pp3$ lying in the ideal of $C$.

One can see this geometrically by considering that divisors linearly equivalent to $D_1$ are cut by $2$-planes in $\Pp3$. Hence, a square root of $\mathcal O(D_1)$ must correspond to a bitangent $2$-plane and the only planes of this form are the tangent planes to the singular quadrics containing $C$.

From this it follows that $\mathcal A$ can be viewed as the moduli space of pairs $(C,\{x_1,\dots,$ $x_4\})$ where the curve $C$ is a smooth genus $1$ curve lying on a fixed quadric cone $Q\subset\Pp3$ and the $x_i$ are $4$ distinct unordered points on $C$ lying on the same plane section $H\subset Q$. If the hyperplane section $H$ is reducible, it is the union of two lines of the ruling of $Q$. In this case, the involution of $C$ interchanging each pair of points lying on the same line of the ruling of $Q$ lifts to an involution of the double cover $\tilde C\rightarrow C$; this involution gives $C'$ a hyperelliptic structure. 
In particular, as shown in \cite[Corollary~1]{cornalba}, the composition of these two involutions on $C$ gives a fixed-point-free involution on $\tilde C$.

Therefore, the coarse moduli space of the closed substack $\mathcal A_h$ of $\mathcal A$ corresponding to bielliptic structures on genus $3$ hyperelliptic curves is isomorphic to the coarse moduli space of $\mathcal D'$.

At this point, it only remains to calculate the cohomology of the complement $\mathcal A_{nh}=\mathcal A\setminus \mathcal A_h$. 
\begin{lemma}
$P_{\mathcal A_{nh}}(t)=L^4t^8+L^3t^7.$
\end{lemma}

\proof
First we note that there is a map $\mathcal A_{nh}\rightarrow\Mm 04/\s_4$ associating to $(C,\{x_i\})$ the configuration $\{x_1,\dots,x_4\}$ of four points on the curve $H\cong\Pp1$. 
To make explicit computations, we view $Q$ as the weighted projective plane $\Pp{}(1,1,2)$ and choose coordinates $u_0,u_1,w$ on $Q$ such that $H$ is defined by the equation $w=0$. 

Every element of $\mathcal A_{nh}$ has an equation of the form
$$
\phi_{\alpha,\epsilon,t}(u_0,u_1,w):=w^2-\alpha(u_0,u_1)w+\epsilon u_0u_1(u_1-u_0)(u_1-tu_0)=0
$$
with $\alpha\in\C[u_0,u_1]_2\cong\C^3$, $\epsilon\in\C$ and $[t]:=\{0,1,\infty,t\}\in\Mm04/\s_4$. This equation is uniquely defined up the action of $\C^*$ on $(\alpha,\epsilon)$ by scaling: 
$$
s(\alpha,\epsilon) = (s^2\alpha,s^4\epsilon).
$$

Hence, we study the incidence correspondence:
$$
\mathcal I:=\left\{(\alpha,\epsilon,[t])\in (\C^4\times\Mm 04)/\s_4\left|\begin{array}{c}
\phi_{\alpha,\epsilon,t}(u_0,u_1,w)=0 \text{  defines}\\
\text{ a nonsingular curve}\end{array}\right.\right\}.
$$

To compute the cohomology of $\mathcal I$, we apply Vassiliev--Gorinov's method to the complement $\Sigma$ of $\mathcal I$ inside $(\C^4\times\Mm04)/\s_4$. 
Specifically, we use the version of the method developed for the case of curves with marked points given in \cite{OTM32}, to which we refer for technical details on the construction. Note that the original construction of Vassiliev--Gorinov's methods is based on the study of Borel--Moore homology (i.e. homology theory with compact support). For stylistic reasons, in this paper we will use cohomology with compact support instead of Borel--Moore homology. All results can be easily adapted by duality.

The first step of Vassiliev--Gorinov's method consists in classifying all possible singular loci of elements of $\Sigma$. This classification is then used to define a cubical space $\mathcal X$ whose geometric realization $|\mathcal X|$ has the same cohomology with compact support as $\Sigma$. This geometric realization has a natural stratification $\{F_i\}$ by locally closed subsets that are indexed by the types of singular sets that arise in the classification. Each stratum $F_i$ can be explicitly described as a bundle over the space 
$$B_i:=\left\{(\phi,[t],K)\left|
\begin{array}{c}(\phi,[t])\in\Sigma, K\subset Q\text{ is a singular configuration}\\
\text{of type $i$ containing the singular locus of }(\phi,[t])\end{array}\right.\right\}.$$
If the configurations of type $i$ are finite sets, then $F_i\rightarrow B_i$ is a nonorientable simplicial bundle; otherwise, the stratum $F_i$ is a union of simplicial bundles.

The Gysin spectral sequence $E^{p,q}_r\Rightarrow \cohc{p+q}{\Sigma}$ with $E_1^{p,q}=\cohc p{F_q}$ associated with the stratification $\{F_i\}$ is called the \emph{Vassiliev's spectral sequence}. 

\begin{itemize}
\item[(1)]
\emph{One singular point on $H$. }
In this case the curve is of the form $w(w+au_0+bu_1)=0$ and both components pass through the singular point. The stratum $F_1$ is a $\C^2$-bundle over $H\times\Mm04/\s_4$.
\item[(2)]
\emph{Two singular points on $H$.}
The curve is of the form $w(w+au_0+bu_1)=0$ and both components have to pass through the singular points. If we fix the two distinct singular points $s_1,s_2$ on $H$, then the $\phi=(\alpha,\epsilon)\in\C^4$ giving a curve singular at $s_1$ and $s_2$ form a $1$-dimensional subspace. 
This yields the following description for the stratum $F_2$: It is  the quotient of a $\C$-bundle over $(0,1)\times(\F2H\times\Mm04)/\s_4$ by the involution 
$(\tau,(s_1,s_2),[t],(\alpha,\epsilon))\mapsto(1-\tau,(s_2,s_1),[t],(\alpha,\epsilon))$. From this it follows that the cohomology with compact support of $F_2$ is concentrated in degree~$7$ and carries a Tate Hodge structure of weight $6$.
\item[(3)]
\emph{The curve $H$.}
The stratum $F_3$ has trivial cohomology with compact support  because $H$ is a smooth rational curve (see e.g. \cite[Lemma 2.19]{OTM4}).
\item[(4)]
\emph{One point $P$ outside $H$.}
Having a singular point outside $H$ imposes three conditions on the equation, hence we get that $F_4$ is a $\C$-bundle over the space of configurations $(P,\{x_1,x_2,x_3,x_4\})$. Note that $P$ is not allowed to lie on the same line of the ruling as any of the $x_i$.
Hence the configuration space is a $\C^*$-bundle over $\Mm 05/\s_4$, whose cohomology with compact support is concentrated in degree $4$ by the results in \cite{getzleroperads}.
\item[(5)]\emph{Two points outside $H$}.
In this case the singular curve is the union of two irreducible plane sections $H_1$, $H_2$ of $Q$ which are different from $H$ and not tangent to each other.
Each of the components passes through exactly two of the $x_i$. Up to reordering the points we may assume that $H_1$ passes through $x_1$ and $x_2$ and $H_2$ passes through $x_3$ and $x_4$.

It is important to observe that such a curve $H_1\cup H_2$ is uniquely identified by the partially ordered configuration $(\{\{x_1,x_2\},\{x_3,x_4\}\},\{s_1,s_2\})$ of points on $H$, where $s_1,s_2$ denote the projections on $H$ of the singular points of the curve. Furthermore, a configuration $(x_1,\dots,x_4,s_1,s_2)$ comes from a singular curve $H_1\cup H_2$ if and only if there is an automorphism of $H\cong \Pp1$ that interchanges the following pairs of points: $x_1\leftrightarrow x_2$, $x_3\leftrightarrow x_4$ and $s_1\leftrightarrow s_2$. This condition defines a codimension $1$ subset $N\subset\Mm06$.

Thus, one can study the cohomology with compact support of the stratum $F_5$ by taking the part of the cohomology with compact support of $N$ such that the symmetric group $\s_4$ acts on the first two points as the representation $\Schur4\oplus \Schur{2,2}$ and the symmetric group $\s_2$ interchanging $s_1$ and $s_2$ acts as the alternating representation. Then a direct computation shows that the only cohomology class with this behaviour is the trivial class. From this it follows that this stratum contributes trivially to the Vassiliev's spectral sequence.

\end{itemize}

From this classification it follows that only configurations of type (1), (2) and (4) contribute to Vassiliev's spectral sequence. 
We give the Vassiliev's spectral sequence associated to this classification of the singularities in Table~\ref{vassiliev-a}. Our description of the singularities also shows that $\Sigma$ has two irreducible components: namely, the divisor $D_1$ of curves with singularities of type (1) and the divisor $D_2$ of curves with singularities of type (3). By the long exact sequence associated to the inclusion $\Sigma\hookrightarrow(\C^4\times\Mm04)/\s_4$, they give two classes $\delta_1,\delta_2$ in the first cohomology group of $\mathcal I$.

\begin{table}
\caption{Spectral sequence converging to $\cohc\pu{(\C^4\times\Mm04)/\s_4\setminus\mathcal I}$.\label{vassiliev-a}}
$$
\begin{array}{r|llll}
q&\\[12pt]
7&\Q(-4)&0&0\\
6&0&0&0\\
5&\Q(-3)&\Q(-3)&\Q(-4)\\
4&0&0&\Q(-3)\\
\hline
&1&2&3&p\\
\text{type}&(1)&(2)&(4)
\end{array}
$$
\end{table}

Consider the Leray spectral sequence associated to the quotient map $q:\;\mathcal I\rightarrow \mathcal A_{nh}$. The action of $\C^*$ on $\mathcal I$ can be extended to an action of $\C$ on $(\C^4\times\Mm04)/\s_4$. In this extended action, one can see that $0\in\C$ maps surjectively to the locus $0\times\Mm04$, which is contained in the intersection of $D_1$ and $D_2$. 
If we denote by $h$ a generator of $\coh1{\C^*}$, this shows that the image of $h$ in the pull-back of the orbit map $\C^*\rightarrow(\C^4\times\Mm04)/\s_4$ is a linear combination of the classes $\delta_1,\delta_2\in\coh1{\mathcal I}$. This can be used to prove that the Leray spectral sequence in cohomology associated to $q$ degenerates at $E_2$, i.e., by Poincar\'e duality, we have $\cohc\pu{\mathcal I}\cong\cohc\pu{\C^*}\otimes\cohc\pu{\mathcal A_{nh}}$ for cohomology with compact support. 

Now let us go back to the spectral sequence in Table~\ref{vassiliev-a}. Since $(\C^4\times\Mm04)/\s_4$ has cohomology with compact support concentrated in degree $10$, the Gysin long exact sequence associated to the inclusion $\Sigma\hookrightarrow(\C^4\times\Mm04)/\s_4$ yields isomorphisms
$$
\cohc{k-1}{\Sigma}\cong\cohc{k+1}{\mathcal I}
$$
for all $k\leq 9$.
In view of the structure of $\cohc\pu{\mathcal I}$ as a tensor product, we have that the $d^1$-differential $E^1_{1,5}\rightarrow E^1_{2,5}$ must be an isomorphism. All other differentials are necessarily $0$ by Hodge-theoretic reasons, since they are maps between pure Hodge structures with different weights. This implies that the cohomology of $\mathcal I$ is isomorphic (as a graded vector space with mixed Hodge structures) to the cohomology of $\C^3\times\C^*\times\C^*$ and that the cohomology of $\mathcal A_{nh}$ is isomorphic to that of $\C^3\times\C^*$. 
\qed

\begin{proof}[of Proposition~\ref{psa}]
We want to compute the cohomology with compact support of $\mathcal A$ by using the  Gysin long exact sequence
$$
\cohc k{\mathcal A}\rightarrow
\cohc k{\mathcal A_h}\xrightarrow{d_k}
\cohc {k+1}{\mathcal A_{nh}}\rightarrow
\cohc {k+1}{\mathcal A}
$$
associated to the inclusion $\mathcal A_h\hookrightarrow\mathcal A$. 
Since the cohomology with compact support of $\mathcal A_{h}$ (resp. $\mathcal A_{nh}$) is nontrivial only in degree $5$ and $6$ (resp. $7$ and $8$) the only differential which may be nontrivial is $d_6$. To prove the claim, we need to show that $d_6$ is an isomorphism or, equivalently, that the cohomology with compact support of $\mathcal A$ vanishes in degree $7$. To prove this, we are allowed to discard in our configurations all subvarieties of $\mathcal A$ of codimension larger than $1$, since they cannot possibly contribute to the cohomology with compact support in such a high degree.

Recall that $\mathcal A$ is the moduli space of pairs $(C,H)$ where $C$ is a smooth genus $1$ curve lying on a fixed quadric cone $Q$ and $H$ is a reduced plane section of $Q$ that intersects $C$ in four distinct points. In particular, the fact that $C$ lies on $Q$ endows $C$ with a natural structure as double cover of $\Pp1$ ramified at $4$ points, giving rise to a natural map $\mathcal A\rightarrow \Mm04/\s_4$. Note that, once a configuration $(0,\infty,1,t)\in\Mm 04$ is chosen, there is a canonical form for the genus $1$ curve in $Q$ lying over $\{0,\infty,1,t\}$, by taking the equation $w^2-u_0u_1(u_0-u_1)(tu_0-u_1)=0$. Therefore, it only remains to describe which reduced plane sections of $Q$ are not tangent to a fixed smooth curve $C\subset Q$. Planes in projective three-space are parametrized by a $\Pp3$; reduced plane sections $H$ come from a rational curve in this $\Pp3$, which we can discard since it has codimension $>1$. The condition that $H$ is not tangent to the fixed $C$ defines an irreducible hypersurface in the $\Pp3$ parametrizing plane sections. Here we only need to deal with irreducible plane sections because the locus of reducible plane sections has already codimension $1$ in $\Pp3$ and any special sublocus of it we could have to discard would not influence $\cohc 7{\mathcal A}$. 

From this description, it follows that the cohomology with compact support of $\mathcal A$ in degree $\geq 7$ coincides with that of the $\s_4$-quotient of a fibration over $\Mm 04$ having the complement of a $\s_4$-invariant hypersurface in $\Pp3$ as fibre. Then the claim follows from the fact that the cohomology with compact support of the complement of an irreducible hypersurfaces is $0$ in degree $6$.
\end{proof}

\subsubsection{The cohomology of the moduli space $\mathcal C'$.}
The space $\mathcal C'$ is the moduli space of genus $1$ curves $C$ with a set $\{x_1,x_2\}$ of (unordered) marked points and a line bundle $L$ such that $L^{\otimes 4}=\mathcal O(x_1+x_2)^{\otimes2}$ but $L^{\otimes 2}\neq\mathcal O(x_1+x_2)$. In this section, we give a more explicit geometric description of $\mathcal C'$ that enables one to compute its cohomology.

First we observe that the line bundle $L^{\otimes4}$ defines an embedding of $C$ into $\Pp3$; in the following, we shall identify $C$ with its image in $\Pp3$. Then all divisors linearly equivalent to $L^{\otimes 4}$ are cut by plane sections of $C$. In particular, the line bundle $L$ itself has to be cut by a plane in $\Pp3$ tangent to $C$ with multiplicity $4$ at one point $q$. This means that there is a quadric cone $Q\subset \Pp3$ containing $C$, whose ruling cuts the divisor $L^{\otimes 2}$ on $C$. Then $q$ is one of the ramification points of this $g^1_2$. Equivalently, the pair $(Q,q)$ identifies the divisor $L$. The points $x_1,x_2$ are contact points of a bitangent plane to $C$. In other words, we can describe $\mathcal C'$ as the moduli space of triples $(C,q,\Pi)$ where $C$ is a smooth genus $1$ curve lying on a fixed quadric cone $Q$, the point $q$ is a ramification point of the $g^1_2$ cut by the ruling of $Q$ and $\Pi\subset\Pp3$ is a plane tangent to $C$ at two distinct points. 

Instead than working directly with $\mathcal C'$, we work with the moduli space $\tilde {\mathcal C}$ of sequences $(C,p_1,p_2,p_3,q,\Pi)$ with $C$ and $\Pi$ are as above and $p_1,p_2,p_3,q$ are the ramification points of the $g^1_2$. The forgetful map $\tilde {\mathcal C}\rightarrow \mathcal C'$ can be interpreted as the quotient by the action of the symmetric group permuting the three points $p_1,p_2,p_3$. There is also a natural  map $\phi:\;\tilde {\mathcal C}\rightarrow\Mm04$ that maps $(C,p_1,p_2,p_3,q,\Pi)$ to the moduli of the ordered branch locus of the $g^1_2$. 

\begin{proof}[of Proposition~\ref{psc'}]
The structure of $\tilde {\mathcal C}$ as an $\s_3$-cover of $\mathcal C'$ ensures that the rational cohomology of the latter space coincides with the $\s_3$-invariant part of the cohomology of $\tilde {\mathcal C}$. We want to compute it by exploiting the Leray spectral sequence associated to the forgetful map $\phi:\;\tilde {\mathcal C}\rightarrow\Mm04$. To this end, we need to describe the fibers of $\phi$ and to calculate their cohomology.

The fibre of $\phi$ over a $4$-tuple $(p_1,p_2,p_3,q)$ is the space of all bitangents to the curve $C$ obtained as a double cover of $\Pp1$ ramified at $p_1,p_2,p_3$ and $q$. We need to parametrize all bitangent planes of the curve $C$ that give rise to reduced plane sections. Since we are studying rational cohomology, which only depends on the coarse moduli space of the stack considered, it is enough to describe all reduced bitangents up to the action of the elliptic involution of $(C,q)$. 

Then an explicit computation yields that the space of all reduced bitangents to $C$ has three distinct irreducible components, isomorphic to $\Pp1$ and depending on the choice of one of the points $p_i$. In particular, the three components are permuted by the $\s_3$-action. In the description of the components, we shall denote the line of the ruling of $Q$ passing through $q$ (respectively, through $p_i$ for $1\leq i\leq 3$) by $\ell$, respectively, $\ell_i$.

Then the point $p_3$ corresponds to the family of bitangents containing the reducible bitangent $\ell_1\cup \ell_2$ and $\ell_3\cup \ell$. This family contains exactly two flex bitangents, i.e. planes tangent to $C$ at one point with multiplicity $4$. The contact points on these flex bitangents are the points of $C$ lying over the points of $C$ lying in the fixed locus of the involution $\ell_1\leftrightarrow\ell_2$, $\ell_3\leftrightarrow\ell$ on the ruling of $Q$. In our description, we have to take only the bitangents with two distinct tangency points $x_1,x_2$, i.e. only proper bitangents, hence we need to discard these two points of the family. Hence each irreducible component of the fibre of $\phi$ over $(p_1,p_2,p_3,q)$ is isomorphic to $\C^*$. If we consider the action of the involution $p_1\leftrightarrow p_2$ on the cohomology of this component of the fibre, we get that it acts trivially in degree $0$ and as the sign representation in degree $1$. 

We obtain the other two components by taking the action of the symmetric group $\s_3$ into account. Then the cohomology of $\phi^{-1}(p_1,p_2,p_3,q)$, with its structure as $\s_3$-representation and its mixed Hodge structures, is given by $\Schur3+\Schur{2,1}$ in degree $0$ and $\Schur{2,1}+\Schur{1^3}$ in degree $1$.

At this point, recall that $\Mm04$ is isomorphic to $\Pp1$ minus $3$ points, and in particular, that its cohomology with the action of $\s_3$ permuting the first three marked points is given by $\Schur3$ in degree $0$ and $\Schur{2,1}$ in degree $1$. The map $\phi$ is $\s_3$-equivariant, hence the $\s_3$-invariant part of the $E_r$ terms of the Leray spectral sequence associated to $\phi$ converges to the cohomology of $\mathcal C'$. From the description the $\s_3$-action on the basis and the fibre of $\phi$, one gets that the only nontrivial $E_2$ terms of this Leray spectral sequence are $(E_2^{0,0})^{\s_3}=\Q$, $(E_2^{1,0})^{\s_3}=\Q(-1)$ and $(E_2^{1,1})^{\s_3}=\Q(-2)$. Then the claim follows by Poincar\'e duality.
\end{proof}

\section{The compactification of the inertia stack of $\mathcal{M}_g$}
\label{compm3}

In Section~\ref{inertiamg} we studied the twisted sectors of $\M g$ as moduli stacks of cyclic covers. In the present section we consider the compactification of these twisted sectors inside the inertia stack of $\Mb g$.
After developing the general theory, we study in detail the compactification of the moduli stacks of admissible covers that correspond to twisted sectors of $\mathcal{M}_3$.
 
Recall that in Section~\ref{inertiamg} we defined the concept of $g$-admissible datum and saw that a $g$-admissible datum $(g',N,d_1,\ldots,d_{N_1})$ always singles out a connected component of the inertia stack, described as a moduli stack of $\mu_N$-ramified covers of curves of genus $g'$. To compactify such moduli stacks of $\mu_N$-covers, 
we rely on the general theory of \emph{twisted stable map}, developed by Abramovich--Vistoli \cite{abvis2}, which in our case specializes to the theory of \emph{admissible covers}, as developed in \cite{acv}. In the language of twisted stable maps, we are studying \emph{balanced twisted stable maps} with value in the trivial gerbe $B \mu_N$. Equivalently, we are studying $\mu_N$-admissible covers \cite[Theorem 4.3.2]{acv}. 

\begin{definition} \label{compactifiedtwisted} Let $A$ be a 
$g$-admissible datum and let us denote as usual the associated moduli stack by $\mathcal{M}_A$, the component consisting of connected covers by $\mathcal{M}_A'$  and the corresponding subgroup of the symmetric group $\s_d$ by $S_A$.
We define $\overline{\mathcal{M}}_A$ (respectively, $\overline{\mathcal{M}}_A'$)  as the closure of $\mathcal{M}_A$ (respectively, $\mathcal{M}_A'$) inside $[{\mathcal{K}}_{g',d} (B \mu_N)/S_A]$, where $\mathcal{K}_{g',d} (B \mu_N)$ is the proper moduli stack of twisted stable $d$-pointed maps of genus $g'$ to $B\mu_N$ defined in \cite[Definition 4.3.1]{abvis2} (see also \cite[Section 2]{acv} for the specific case of $B \mu_N$).
\end{definition}

The stacks $\overline{\mathcal M}_A$ give connected components of
 the inertia stack of $\overline{\mathcal{M}}_g$ by associating to each admissible cover in $\overline{\mathcal M}_A$ the pair $(C^{\operatorname{stab}},\phi)$ where $C^{\operatorname{stab}}$ is obtained by stabilizing the source curve $C$ of the cover, and $\phi$ is the automorphism on $C^{\operatorname{stab}}$ induced by the action of $\mu_N$ on $C$. 
Conversely, since the moduli space of admissible covers is proper and contains the smooth ones, each smoothable cyclic cover $X\rightarrow C$ where $X$ is a stable curve has an associated admissible cover, obtained by repeatedly blowing up the two curves $X$ and $C$.

\begin{proposition} \label{instack2} Let us fix $g,N >1$. Then the compactified inertia stack $I_N(\mathcal{M}_g)$ of Definition~\ref{ibar} is isomorphic to the disjoint union of all nonempty stacks $\overline{\mathcal{M}}_A'$ 
for all $g$-admissible data $A=(g',m;d_1,\ldots,d_{m-1})$ with order $m$ equal to $N$.
\end{proposition}

\begin{proof} The proof of the proposition follows by adapting the proof of Proposition \ref{corrispondenza}. 
In the case when the covering curve $C'$ turns out to be unstable, one applies the usual stabilization procedure (\cite{knudsen}).
\end{proof}

\noindent In other words, in Definition \ref{compactifiedtwisted}, we have described all the twisted sectors of $I(\overline{\mathcal{M}}_g)$ that do not come from the boundary. 

\begin{remark} It is clear that $\overline{I}(\mathcal{M}_g) \neq I(\overline{\mathcal{M}}_g)$, i.e. that there are twisted sectors of $\M g$ that do not contain any smooth curves. To see this, take two smooth, $1$-pointed curves, one of genus $g'\geq1$ and the other of genus $g-g'$, each admitting an automorphism of different order that fixes the marked point. Now the curve $C$ obtained gluing the two curves at the marked points, with the automorphism induced by the two automorphisms, is a point in the inertia stack of $\overline{\mathcal{M}}_g$ that is not in $\overline{I}(\mathcal{M}_g)$.
\end{remark}

Next, we turn our attention to the cohomology of the moduli stacks $\overline{\mathcal{M}}_A$ and, more specifically, to the compactified twisted sectors of $\M g$ whose general object is a curve that is described as the cyclic cover of a genus $0$ curve. 
For combinatorial reasons, the large majority of cases fall into this class.
We can reduce the problem of computing the cohomology groups of the $\overline{\mathcal{M}}_A$ with $g'=0$  to the problem of computing the part of the cohomology of $\overline{\mathcal{M}}_{0,d}$ under the action of the subgroup $S_A\subset \s_n$, which is then known (see \cite[5.8]{getzleroperads}). This relies on the construction of the $\overline{\mathcal{M}}_A$ as stack quotients of a connected substack of ${\mathcal{K}}_{g',d} (B \mu_N)$.

If $X$ is a scheme, $D$ is an effective Cartier divisor, and $r$ is a natural number, then \cite{cadman} and \cite{agv2} introduced the stack $X_{D,r}$, called the \emph{root of a line bundle with a section}. The following result is essential for our application:
\begin{proposition} (\cite[Corollary 2.3.7]{cadman}) Let $X$ be a scheme. If $X_{D,r}$ is obtained from $X$ by applying the root construction, the canonical map $X_{D,r} \to X$ exhibits $X$ as the coarse moduli space of $X_{D,r}$.
\end{proposition}

\begin{theorem} (\cite[p.2]{bayercadman}) \label{bayercadman} Let $A$ be a $g$-admissible datum (see Definition \ref{admissible}), with $g'$ equal to $0$. The space $\overline{\mathcal{M}}_A$ is then a $\mu_N$-gerbe over the quotient stack $[X / S_A]$, where $X$ is a stack constructed starting from $\overline{\mathcal{M}}_{0,\sum d_i}$ by successively applying the root construction (see \cite[Section 2]{bayercadman}).
\end {theorem}
\noindent By combining these two results, we obtain a simple description of the cohomology of the twisted sectors of $\overline{I} (\mathcal{M}_g)$ whose general element is a cyclic cover of a genus $0$ curve:
\begin{corollary} \label{corollariocadman} If $A$ is a $g$-admissible datum with base genus $g'$ equal to $0$, then the stack $\overline{\mathcal{M}}_A$ has the same rational Chow groups and rational cohomology groups as $\overline{\mathcal{M}}_{0,\sum d_i}\big/ S_A $.
\end{corollary}

Using the results of \cite[5.8]{getzleroperads}, we can now determine the rational cohomology of the positive-dimensional twisted sectors $\overline{\mathcal{M}}_A'$ whose general object covers a genus $0$ curve. The Hodge--Grothendieck characters of these spaces is listed in the third column of tables \ref{base0dim>0} and \ref{base0dim0}.  The twisted sectors are proper smooth stacks, hence their $k$-cohomology group carries a pure Hodge structure of weight $k$. For this reason, the Hodge--Grothendieck characters determine the rational cohomology as vector space with Hodge structures.

\subsection{The compactification of the inertia stack of $\mathcal{M}_3$} 
There are only four twisted sectors in ${\mathcal{M}}_{3}$ whose general object covers curves of genus $1$ or $2$. They are the spaces that we called $\mathcal{A}, \mathcal{B}, \mathcal{C}'$ and $\mathcal{D}'$ in section~\ref{inertiam3}. The remainder of the present section is thus devoted to investigating the geometry of their compactifications, in order to compute their rational cohomology.

The general strategy here is the following. Since we deal with proper smooth stacks, their cohomology is determined uniquely by the Hodge--Gro\-then\-dieck character. To compute this, we exploit the additivity of Hodge--Grothendieck characters for compact support.

As we already know the Hodge--Grothendieck characters for compact support of the open parts as a consequence of the Propositions~\ref{basegenus0}, \ref{psa}, \ref{psc'} and the Corollaries~\ref{psb} and~\ref{psd'},
we need to study the irreducible components of 
$$\overline{\mathcal{A}} \setminus \mathcal{A}, \ \overline{\mathcal{B}} \setminus \mathcal{B}, \
\overline{\mathcal{C}'} \setminus \mathcal{C}',\ \overline{\mathcal{D}'} \setminus \mathcal{D}'.$$ 

Furthermore, by Poincar\'e duality, all we need to know are the coefficients of the Hodge--Grothendieck character with degree greater than or equal to half the complex dimension of the stack considered.

We describe in detail the case of $\mathcal{A}$ (the most complicated), and we sketch the proofs of the other cases.

\begin{proposition} \label{psca} The Hodge--Grothendieck character of $\overline{\mathcal{A}}$ is
$$
\chihdg{\overline{\mathcal A}}=L^4+6 L^3+ 9L^2+ 6L+ 1.
$$
\end{proposition}

\begin{proof} We have seen in Proposition \ref{psa} that the Hodge--Grothendieck character for compact support of $\mathcal{A}$ is $L^4-L^2$. The moduli stack $\mathcal{A}$ admits a finite \'etale map onto $[\mathcal{M}_{1,4}/\s_4]$. 
This map extends to a finite map $t: \overline{\mathcal{A}} \to [\overline{\mathcal{M}}_{1,4}/\s_4]$ (see \cite[3.0.5]{acv}) on the compactification $\overline{\mathcal{A}}$ by means of admissible covers. 

The stratification of $\Mmb 14$ by topological type induces a stratification on $\overline{\mathcal{A}}$. We need to study its strata of codimension $1$ and $2$.
The quotient stack $[\overline{\mathcal{M}}_{1,4}/\s_4]$ has four boundary divisors, and their general element is as in Figure~\ref{fig:bdm14}. We denote by $D_1,\dots,D_4$ the associated locally closed codimension $1$ strata, obtained by removing all curves with more than one node.

\begin{figure}[ht]
\centering
\tiny
\psfrag{1}{$1$} 
\psfrag{2}{$0$}
\begin{tabular}{cccc}
\includegraphics[scale=0.3]{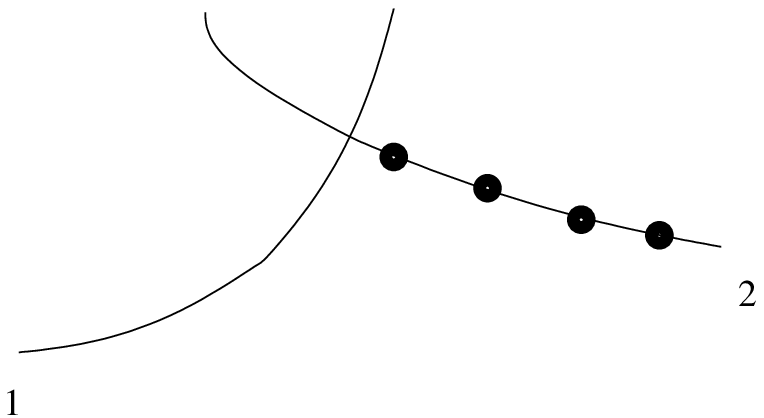}&
\includegraphics[scale=0.3]{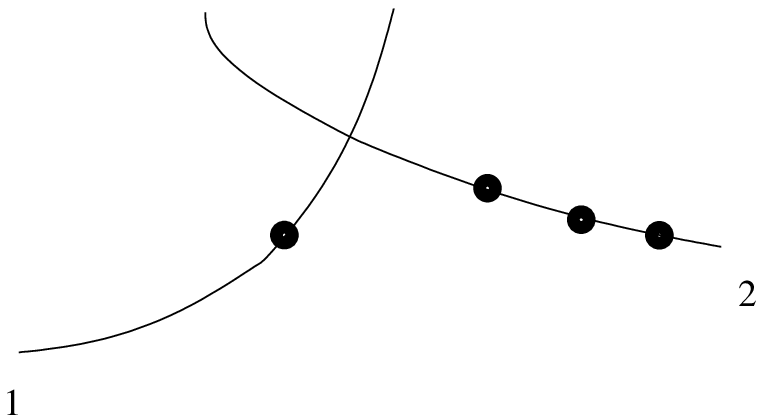}&
\includegraphics[scale=0.3]{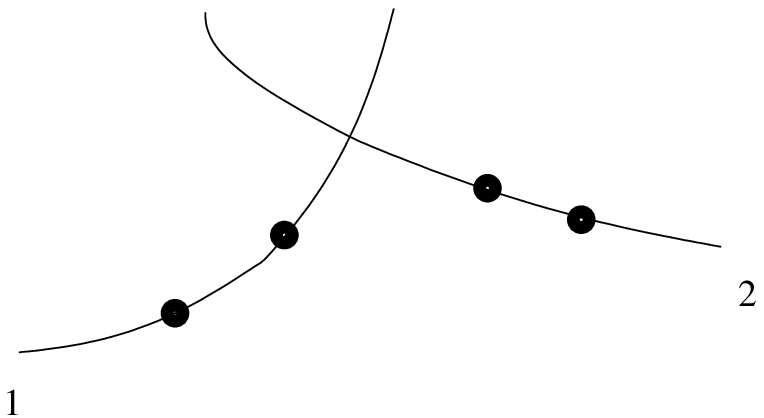}&
\includegraphics[scale=0.3]{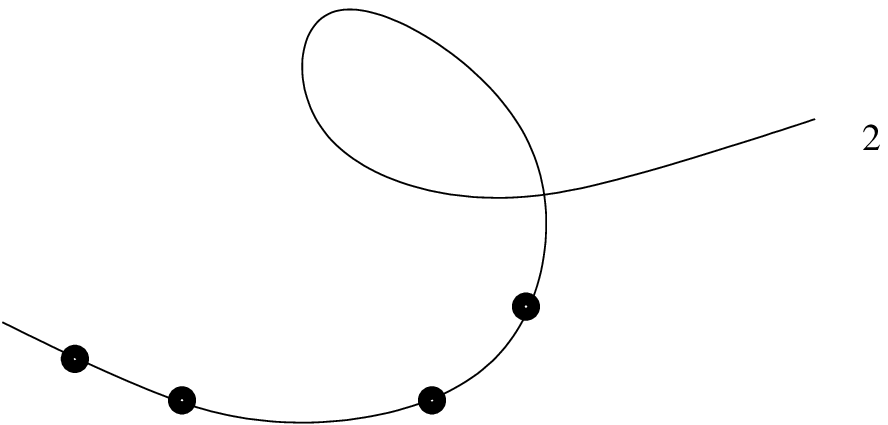}
\\
$D_1$&$D_2$&$D_3$&$D_4$
\end{tabular}
\caption{\label{fig:bdm14} The four boundary strata of codimension $1$ in $\big[\overline{\mathcal{M}}_{1,4}/\s_4 \big]$}
\end{figure}

Now we describe one by one the irreducible components of $\overline{\mathcal{A}}$ that map onto the four codimension $1$ boundary strata we have just pictured:
\begin{enumerate}
\item There are two irreducible components $D_1', D_1''$ lying over $D_1$. They pa\-ram\-e\-trize admissible double covers $C\rightarrow C'$ such that the restriction to the preimage of the genus $1$ component is, respectively, a trivial $\mu_2$-cover in the case of $D_1'$ and a nontrivial $\mu_2$-cover in the case of $D_1''$.  The coarse moduli space of $D_1'$ is isomorphic to $\mathcal{M}_{1,1} \times \mathcal{M}_{0,5}/\s_4$, and its Hodge--Grothendieck character for compact support is then $L^3$. The moduli space of $D_1''$  is isomorphic to $X_1(2) \times  \mathcal{M}_{0,5}/\s_4$ and $\chihdg{D_1''}=L^3-L^2$.
\item Over $D_2$  there is one irreducible component $D_2'$, whose moduli space is isomorphic to $II_1 \times \mathcal{M}_{0,4}/\s_3$. Here $II_1$ is the moduli space of bielliptic curves with a choice of a distinguished bielliptic involution, and an ordering of the ramification points (see \cite[Proposition 2.25]{pagani2}). 
As shown there, the stack $II_1$ has $\Mm05/\s_3$ as coarse moduli space. Hence, one has $\chihdg{D_2'}=L^3-L^2+L$;
\item Over $D_3$ there is one component $D_3'$. 
Its moduli space is $II^1/\s_2$, where $II^1$ is the moduli space of bielliptic curves $C$ with a distinguished bielliptic involution $\alpha$ and a point $p$ not fixed by $\alpha$, and the involution on $II^1$ is defined by  sending $(C,p, \alpha)$ to $(C, \alpha(p), \alpha)$. 
Using the construction of \cite[Proposition~2.22]{pagani2}, it is easy to show that the coarse moduli space of $II^1/\s_2$ is a $\C^*$-bundle over $\Mm05/\s_3$. 
In particular, one has $\chihdg{D_3'}=\chihdg{II^1/\s_2}=L^3-2L^2+2L-1$;
\item Finally, over $D_4$ there are two components, one whose general element is a cover unramified over the node, and the other one whose general element is totally ramified over the node. Both these moduli spaces are isomorphic to $\mathcal{M}_{0,6}/\s_4 \times \s_2$, and their Hodge--Grothendieck character for compact support is then $L^3-L^2$.
\end{enumerate}

The second step is to study the number of irreducible components of
the preimages in $\overline{\mathcal{M}}_A$ of each of the codimension~$2$ strata of $[\overline{\mathcal{M}}_{1,4}/\s_4]$.  There are exactly nine codimension~$2$ strata $F_1,\dots,F_9$ in $[\Mmb14/\s_4]$. We describe them in Figure~\ref{cod2mb14} by drawing their general element.

\begin{figure}[ht]
\centering
\tiny
\psfrag{1}{$1$} 
\psfrag{2}{$0$}
\begin{tabular}{ccc}
\includegraphics[scale=0.4]{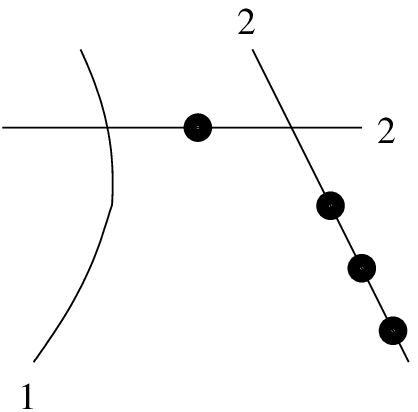}&
\includegraphics[scale=0.4]{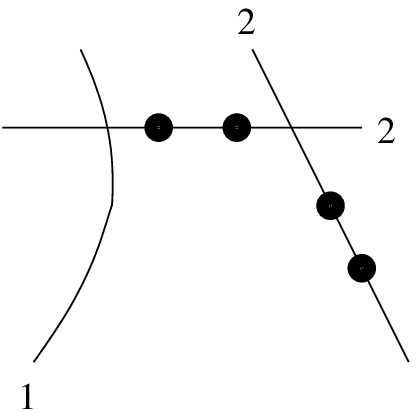}&
\includegraphics[scale=0.4]{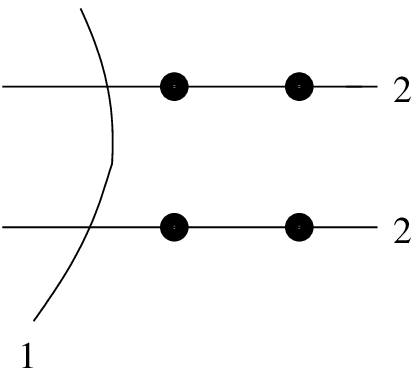}\\
$F_1$&$F_2$&$F_3$\\
\\
\includegraphics[scale=0.4]{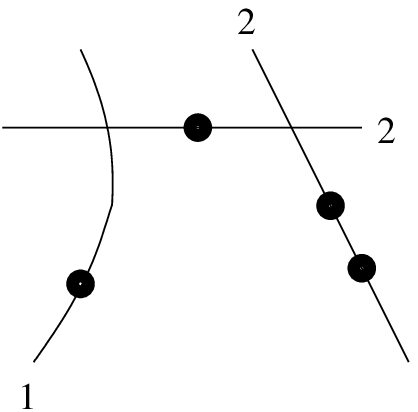}&
\includegraphics[scale=0.4]{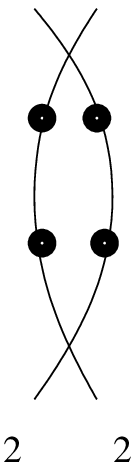}&
\includegraphics[scale=0.4]{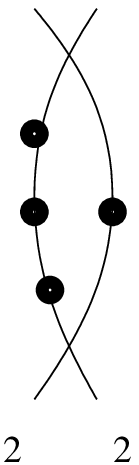}
\\
$F_4$&$F_5$&$F_6$\\
\includegraphics[scale=0.4]{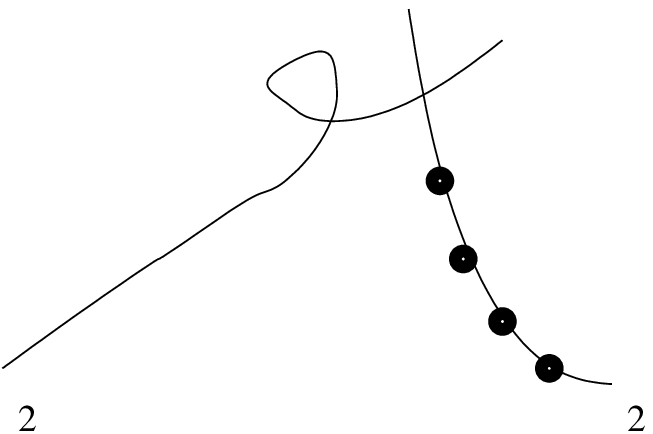}&
\includegraphics[scale=0.4]{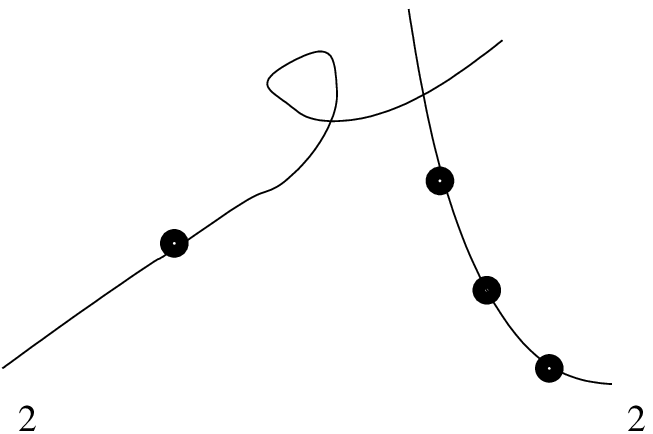}&
\includegraphics[scale=0.4]{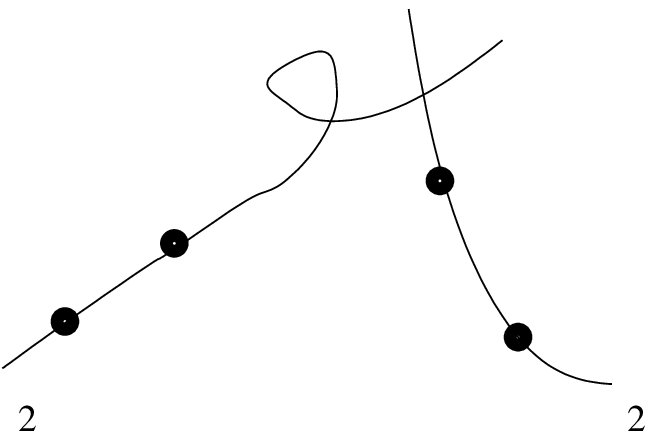}
\\
$F_7$&$F_8$&$F_9$
\end{tabular}
\caption{\label{cod2mb14}The nine strata of codimension~$2$ in $\big[\overline{\mathcal{M}}_{1,4}/\s_4 \big]$}
\end{figure}

\begin{itemize}
\item The strata $F_1, F_2$ and $F_3$ lie in the closure of the codimension $1$ stratum $D_1$. Again, the preimage of each of them under $t$ has two components, corresponding respectively to trivial and nontrivial $\mu_2$-covers of the irreducible component of genus $1$;
\item The stratum $F_4$ is contained in the closure of $D_2$ and its preimage under $t$ is irreducible;
\item The preimage $t^{-1}(F_5)$ has two components, 
corresponding to whether the two nodes are contained or not in the branch locus;
\item The preimage $t^{-1}(F_6)$ is irreducible;
\item The strata $F_7, F_8$ and $F_9$ lie in the closure of $D_4$ and the preimage of each of them has $2$ irreducible components. 
Indeed, while the fact that the separating node is or not part of the branch locus depends on the disposition of the other branch points, the irreducible node can be either a branch point or a regular point, which gives rise to two different components, exactly as it was the case for $D_4$.
\end{itemize}
So, putting everything together, we have:
\begin{multline}
\chihdg{\overline{\mathcal{A}}}= \chihdg {\mathcal{A}} + \chihdg {\textrm{codim }1} + \chihdg{\textrm{codim } 2} + \dots\\= L^4-L^2 + (6 L^3- 6 L^2) + (16 L^2) + \dots 
\end{multline}
and now, using the fact that Poincar\'e duality holds, we get to the conclusion:
\begin{equation}
\chihdg{\overline{\mathcal{A}}}=  L^4+ 6L^3+ 9 L^2 + 6L +1.
\end{equation}
\end{proof}

\begin{remark}
In the proof of the above result we have used Poincar\'e duality. It is possible to reach the same conclusion by means of a more refined analysis of the boundary strata. This gives a non-trivial check on the whole result.

More precisely, the quantity $L^4+ 6L^3+ 9 L^2 + 6L +1$ is obtained as the sum of the following contributions:
$$
(L^4-L^2)+(6L^3-6L^2+3L-1)+(16L^2-15L+5)+(18L-13)+10,
$$
where we have collected the terms according to the codimension of the corresponding stratum (from codimension $0$ to codimension $4$).
\end{remark}

\begin{proposition} \label{pscb} The Hodge--Grothendieck character of $\overline{\mathcal{B}}$ is
$
L^2+3L +1
$.
\end{proposition}

\begin{proof}[Sketch]
We have seen in Corollary~\ref{psb} that  $\chihdg{\mathcal{B}}$ equals $L^2-L+1$. The moduli stack $\mathcal{B}$ admits a finite \'etale map onto $\mathcal{M}_{1,2}$. On its compactification $\overline{\mathcal{B}}$ by means of admissible covers this map extends to a finite map $t: \overline{\mathcal{B}} \to \overline{\mathcal{M}}_{1,2}$. 
The boundary $\partial \mathcal{M}_{1,2}$ has two irreducible components: namely, a component $\delta_1$ whose general element is of compact type, and another component $\delta_0$  whose general element is an irreducible curve of geometric genus $0$ with one node. It is easy to see that the preimage of both $\delta_0$ and $\delta_1$ under $t$  has two irreducible components.
By the additivity of Hodge--Grothendieck characters for compact support this ensures that the coefficients of $L$ and $L^2$ in the Hodge--Grothendieck character of $\overline{\mathcal B}$ are as in the claim. The constant coefficient is equal to the degree $2$ coefficient by Poincar\'e duality.
\end{proof}

\begin{proposition} \label{pscc'} The Hodge--Grothendieck character of $\overline{\mathcal{C}}'$ is
$
L^2+3L +1
$.
\end{proposition}

\begin{proof}[Sketch]
We have seen in Proposition \ref{psc'} that  $\chihdg{\mathcal{C}'}$ equals $L^2-L+1$. The moduli stack $\mathcal{C}'$ admits a finite \'etale map onto $[\mathcal{M}_{1,2}/\s_2]$. 
This map extends to a finite map $t: \overline{\mathcal{C}'} \to [\overline{\mathcal{M}}_{1,2}/\s_2]$ on the compactification $\overline{\mathcal{C}'}$ by means of admissible covers. 
There are two irreducible components of codimension $1$ in $[\partial \mathcal{M}_{1,2}/\s_2]$, one whose general element is of compact type, and the other one whose general element is an irreducible curve of geometric genus $0$ with one node. We call them respectively $\delta_1$ and $\delta_0$. Now the fiber of $t$ over $\delta_1$ is made of one irreducible component, while the fiber of $t$ over $\delta_0$ is made of three irreducible components, one for each possible way of prescribing a balanced ramification over the node.
\end{proof}

\begin{proposition}\label{pscd'} The stack $\overline{\mathcal{D}'}$ has Hodge--Grothendieck character for compact support 
$
L^3+4L^2 +4L +1.
$
\end{proposition}
 
One can compute the Hodge--Grothendieck character of $\overline{\mathcal{D}'}$ with the same strategy used for the Propositions~\ref{psca}--\ref{pscc'}, but also a more geometric proof is possible:
\begin{remark} As shown by Bernstein in \cite[Sections 2.5--2.6]{bernstein}, 
the coarse moduli space of admissible \'etale double covers in genus $g$ coincides with the coarse \emph{moduli space of Prym curves}, i.e. of quasistable curves together with a line bundle which is isomorphic to $\mathcal O(1)$ on the exceptional components and to a square root of the trivial bundle on the nonexceptional ones. 
Indeed the latter coarse moduli space, usually denoted by $\overline{R}_g$, has been extensively studied. The results of Bernstein imply that $\overline{\mathcal{D}'}$ has coarse moduli space $\overline{R}_2$, and by \cite[Lemma 20]{sebastian} the latter coincides with $\overline{\mathcal{M}}_{0,6}/\s_4 \times \s_2$.
The combination of these results implies Proposition \ref{pscd'}.
\end{remark}

\section {The Age Grading}
\subsection {Definition of Chen--Ruan degree}
To complete our computation of the orbifold cohomology of $\M3$ we still need to consider its structure as a $\Q$-graded vector space. To compute the new grading, the degree of each cohomology class of the inertia stack of $\M3$ has to be shifted by the \emph{age} of the twisted sector.
We recall from the introduction that the motivation for the new grading is that orbifold cohomology can be endowed with a natural product, the Chen--Ruan product, which is not compatible with the ``naive'' grading of the cohomology of the inertia stack, but turns it into a $\Q$-graded algebra, once it is endowed with the new grading.
 
In the following, we denote by $R{\mu_N}$  the representation ring of $\mu_N$, and by $\zeta_N$ a chosen generator for the group $\mu_N$ of $N$-roots of $1$. 
\begin {definition} (\cite[Section 7.1]{agv2}) Let $\rho:\mu_N \to \mathbb{C}^*$ be a group homomorphism. It is determined by an integer $0 \leq k \leq N-1$ as $\rho( \zeta_N)= \zeta_N^k$. We define the \emph{age} of $\rho$ by:
$$
\textrm{age}(\rho)=k/N.
$$
The age extends to a unique additive homomorphism
$
\textrm{age}: R \mu_N \to \mathbb{Q}
$.
\end {definition}
\noindent Next, we define the age of a twisted sector $Y$. In the following definition, we let $f$ be the restriction to the twisted sector $Y$ of the natural map $I(X) \to X$.
\begin{definition} \label{definitionage} (\cite[Section 3.2]{chenruan}, \cite[Definition 7.1.1]{agv2}) Let $Y$ be a twisted sector and $g: \spec \mathbb{C} \to Y$ a point. Then the pull-back via $f\circ g$ of the tangent sheaf, $(f\circ g)^*(T_X)$, is a representation of $\mu_N$ on a finite-dimensional vector space.  We define: $$a(Y):= \textrm{age}((f\circ g)^*(T_X))$$
\end{definition}
\noindent We are ready to define the orbifold, or Chen--Ruan, degree.

\begin{definition} \label{defcoomorb2} (\cite[Definition 3.2.2]{chenruan}) We define the $d$th degree orbifold cohomology group of $X$ as follows:
$$
H^d_{CR}(X):= \bigoplus_Y  \coh{d-2 a(Y)}Y
$$
where the sum runs over all components sectors $Y$ of $I(X)$.
\end{definition}

\begin{definition} \label{pp1} We define the \emph{orbifold Poincar\'e polynomial} of $X$ as:
$$
P^{CR}(X):= \sum_{i \in [0, \dim_{\mathbb{C}}(X)] \cap \mathbb{Q}} \dim H^{i}_{CR}(X) t^i
$$
Note that the degree of $H^\pu_{CR}$ is given by the unconventional grading defined in Definition~\ref{defcoomorb2}.
\end{definition}

In Definition \ref{ibar} we have introduced a compactification of the inertia stack. The connected components of $\overline{I}(X)$ can be assigned the age grading as in this section, simply using the fact that every connected component of $\overline{I}(X)$ contains, by its very definition, a connected component of ${I}(X)$. Taking cohomology,  one obtains a linear subspace of the orbifold cohomology ${H}^*_{CR}(\overline{\mathcal{M}}_g)$. 
Although we shall not deal with this in the present paper, it is actually possible to prove that the orbifold cohomology classes coming from $\overline{I}(X)$ form a subalgebra of the Chen--Ruan cohomology ring.

\begin{definition} \label{pp2} We define the compactified orbifold cohomology $\overline{H}^\pu_{CR}(X)$ as $H^\pu(\overline{I}(X))$ with the grading induced from $H^\pu_{CR}(X)$. The additive structure of this vector space can be recollected in the polynomial:
$$
\overline{P}^{CR}(X):= \sum_{i \in [0, \dim_{\mathbb{C}}(X)] \cap \mathbb{Q}} \dim \overline{H}^{i}_{CR}(X) t^i
$$
We call this polynomial the \emph{compactified orbifold Poincar\'e polynomial}.
\end{definition}

We observe that Poincar\'e duality holds for the orbifold cohomology of $X$ if $X$ is a proper smooth stack, and for the compactified orbifold cohomology of $X$ with respect to the compactification $X\subset\overline X$, if $\overline X$ is smooth.
\subsection {The orbifold Poincar\'e polynomials of $\mathcal{M}_3$}
\label{crpp3}

To compute the age of the twisted sectors, we will use the following Proposition, suggested to us by Fantechi \cite{fantechi}, which builds on \cite[Proposition 4.1]{pardini}.

\begin{proposition} \label{fantechiage} (\cite{fantechi}) Let $g>1$ and let $Y$ be the twisted sector of $\mathcal{M}_{g}$ corresponding to the discrete datum $(g',N,d_1,\ldots,d_{N-1})$ (Definition \ref{admissible}, Proposition \ref{corrispondenza}). Then its age is equal to:
\begin{equation}\label{etamg}
a(Y) = \frac{(3 g'-3)(N-1)}{2} + \frac{1}{N} \sum_{i=1}^{N-1} d_i \sum_{k=1}^{N-1} k \left( \left\{\frac{k i}{N}\right\} + \sigma(k,i) \right)
\end{equation}
where $\sigma(k,i)=0$ if $ki+\gcd(i,N)\equiv 0 \pmod N$ and $1$ otherwise.

\end {proposition}

\begin{proof}
A point of $Y=\mathcal{M}_{(N,g',d_1, \ldots, d_{N-1})}'$ is a $\mu_N$-cover $C \to C'$. By standard infinitesimal deformation theory, the fiber of the tangent sheaf $T \mathcal{M}_g$ at $C$ is $H^1(C, T_C)$. Thus the cyclic group $\mu_N$ acts on the vector space $H^1(C, T_C)$, and the latter splits in a direct sum of eigenspaces as
$$
H^1(C, T_C)= \bigoplus H^1(C,T_C)^{\chi_k},
$$
where  $H^1(C,T_C)^{\chi_k}$ denotes the subspace where $\mu_N$ acts with weight $k$.
According to Definition \ref{definitionage}, we have:
\begin{equation} \label{formula1}
a(Y)= \sum_{k=1}^{N-1} \frac{k}{N} h^1(C, T_C)^{\chi_k}.
\end{equation}
Now by stability we can substitute $h^1(C,T_C)$ with $-\chi(C, T_C)$. Moreover, since $\pi$ is finite and $T_C$ is coherent, we have that all higher invariant direct images $R^i \pi_*^{\mu_N}(T_C)$ vanish (see \cite[Chapter V, Corollary p. 202]{grot}), and thus: 
$$H^i(C, T_C)^{\chi_k}=H^i(C', (\pi_* T_C) )^{\chi_k}= H^i(C', (\pi_* T_C)^{\chi_k} ),  \quad i=0,1.$$ 
So \eqref{formula1} becomes
\begin{equation} \label{formula2}
a(Y)= \sum_{k=1}^{N-1} -\frac{k}{N} \chi((\pi_*T_C)^{\chi_k}).
\end{equation}
The sheaf $\pi_*(T_C)$ is studied in \cite[Proposition 4.1,(a)]{pardini}. In particular, one has
\begin{equation} \label{formula3}
\deg\left(\pi_* (T_C)^{\chi_k}\right)= \deg T_{C'} - \sum_{\sigma(k,i)=1} d_i - \deg L_k
\end{equation}
where the degree of the line bundles $L_k$ can be computed via \cite[Proposition 2.1]{pardini} (see also \eqref{ellea}):
$$
\deg L_k = \sum_i \left\{ \frac{k i}{N}\right\} d_i.
$$

Now, if we combine~\eqref{formula3} with Riemann-Roch to compute the Euler characteristic in~\eqref{formula2}, we obtain 
\begin{equation}\label{rilevante}
a(Y)=\frac{1}{N} \sum_{k=1}^{N-1} k \left( 3g' -3 + \sum_{\sigma(k,i)=1} d_i + \sum_{i=1}^{N-1} \left\{ \frac{k i}{N}\right\} d_i \right)
\end{equation}
and from this last equation the statement follows.
\end{proof}
Note that by combining Proposition \ref{fantechiage} with \cite[Lemma 4.6]{pagani2} (respectively, with \cite[Corollary 4.10]{pagani2}), one obtains a closed formula for the age of the twisted sectors of $\mathcal{M}_{g,n}$ (respectively, $\mathcal{M}_{g,n}^{rt}$, the moduli space of stable curves that have one smooth component of geometric genus $g$).

The results we have seen so far enable us to compute the orbifold Poincar\'e polynomial of $\mathcal{M}_3$ and its compactified orbifold Poincar\'e polynomial.

\begin{theorem} \label{pp1m3} The orbifold Poincar\'e polynomial (Definition \ref{pp1}) of $\mathcal{M}_3$ is:
\begin{align*}
1+t+2t^2+t^3+t^{\frac{10}{3}}+t^{\frac{7}{2}}+4t^4+2 t^{\frac{9}{2}}+ 2 t^{\frac{14}{3}}+ t^{\frac{33}{7}}+ 5 t^5+ t^{\frac{46}{9}}+ t^{\frac{36}{7}}
\\+ 3t^{\frac{16}{3}}+ t^{\frac{38}{7}}+ 4 t^{\frac{11}{2}}
+ t^{\frac{50}{9}}
+ t^{\frac{39}{7}}
+t^{\frac{17}{3}}+ t^{\frac{40}{7}}+ t^{\frac{52}{9}}
+ t^{\frac{41}{7}}+10 t^6 +t^{\frac{43}{7}}\\
+ t^{\frac{56}{9}}+ t^{\frac{44}{7}}+ t^{\frac{19}{3}}+ t^{\frac{45}{7}}+ t^{\frac{58}{9}}+3 t^{\frac{13}{2}}+ t^{\frac{46}{7}}+ 2 t^{\frac{20}{3}}+ t^{\frac{48}{7}}+ t^{\frac{62}{9}}+ t^{\frac{51}{7}}.
\end{align*}
\end{theorem}

\begin{proof} The ordinary Poincar\'e polynomial of $\mathcal{M}_3$ is calculated in \cite[(4.7)]{loo}. Then the result follows from the determination of the Poincar\'e polynomials of the moduli stacks $\mathcal{M}_A'$  carried out in Section \ref{inertiam3}  for any $3$-admissible datum $A$. See in particular Table~\ref{base0dim>0} and~\ref{base0dim0}, Proposition~\ref{psa} and \ref{psc'} and Corollary~\ref{psb} and \ref{psd'}.

Finally, the age of the twisted sectors is computed using Proposition \ref{etamg}.
\end{proof}

\begin{remark}
In our work we have computed the Hodge structures on the twisted sectors of $\M3$, which turned out to be always pure and of Tate type. Therefore, if following \cite[Def. 3.2.4]{chenruan} one defines  $$H^{p,q}(H_{CR}^d(\M3))=
\bigoplus_Y  H^{p-a(Y),q-a(Y)}(\coh{d-2a(Y)}Y),
$$
where the sum runs over all sectors $Y$ of $\M3$, one can endow $H_{CR}^d(\M3)$ with a canonical structure of direct sum of pure Hodge structures of Tate type whose Hodge filtrations have been shifted by a rational number. Under this convention, what we actually prove is the following result:
\begin{align*}
&{\sum_{p,d\in\Q}\dim(H^{p,p}(H_{CR}^d(\M3)))L^{2p}t^d=} \\
 1&+ L^{1/2}t + 2 Lt^2 + L^{3/2}t^3 + (3L^2+L^{5/2})t^4\\
& + (2L^{5/2}+3L^3)t^5+(5L^3+2L^{7/2}+2L^4+L^6)t^6
\\
&+ \Lt{7/4}{7/2}(1+Lt) + \Lt{9/4}{9/2}(1+Lt) 
+ 3\Lt{11/4}{11/2} + 3\Lt{13/4}{13/2} 
\\
&+ \Lt{5/3}{10/3} + \Lt{7/3}{14/3}(2+Lt)+ \Lt{8/3}{16/3}(3+Lt)+2\Lt{10/3}{20/3}
\\
&+ \Lt{33/14}{33/7} + \Lt{18/7}{36/7}  + \Lt{19/7}{38/7}+ \Lt{39/14}{39/7} + \Lt{20/7}{40/7}    
\\
&+ \Lt{41/14}{41/7}+ \Lt{43/14}{43/7}  + \Lt{22/7}{44/7} + \Lt{45/14}{45/7}  + \Lt{23/7}{46/7} 
\\  
& + \Lt{24/7}{48/7}+ \Lt{51/14}{51/7} + \Lt{23/9}{46/9}  + \Lt{25/9}{50/9} 
\\
& + \Lt{26/9}{52/9}  + \Lt{28/9}{56/9}  + \Lt{29/9}{58/9}  + \Lt{31/9}{62/9}.
\end{align*}
\end{remark}

\begin{theorem} \label{pp2m3} The compactified orbifold Poincar\'e polynomial (see Definition~\ref{pp2}) of ${\mathcal{M}}_3$ is:
\begin{align*}
1& + t + 4 t^2 + 4 t^3 + t^{\frac{10}{3}} + t^{\frac{7}{2}} + 16 t^4 + 
      t^{\frac{9}{2}} + 2 t^{\frac{14}{3}} + t^{\frac{33}{7}} + 12 t^5 \\&+ 
      t^{\frac{46}{9}} + t^{\frac{36}{7}} + 5 t^{\frac{16}{3}} 
+ t^{\frac{38}{7}} + 
      5 t^{\frac{11}{2}} + t^{\frac{50}{9}} + t^{\frac{39}{7}}  +t^{\frac{40}{7}} + 
      t^{\frac{52}{9}} + t^{\frac{41}{7}}  \\&+ 31 t^6+ t^{\frac{43}{7}} + t^{\frac{56}{9}}+      t^{\frac{44}{7}} + t^{\frac{45}{7}} + t^{\frac{58}{9}} + 5 t^{\frac{13}{2}} 
+       t^{\frac{46}{7}} + 5t^{\frac{20}{3}} + t^{\frac{48}{7}} + t^{\frac{62}{9}} \\&+ 
      12 t^7 + t^{\frac{51}{7}} +2 t^{\frac{22}{3}} + t^{\frac{15}{2}} + 
      16 t^8 + t^{\frac{17}{2}} + t^{\frac{26}{3}} + 4t^9 + 4 t^{10} + 
      t^{11} + t^{12}.
\end{align*}
\end{theorem}

\begin{remark}
From the description of the cohomology of the compactified twisted sectors it also follows that the whole compactified orbifold cohomology of $\M3$ is additively generated by algebraic classes.
\end{remark}
\begin{proof}
The ordinary Poincar\'e polynomial of $\overline{\mathcal{M}}_3$ was first computed in \cite[Prop.~16]{getzlertopolo}.
The proof of this theorem follows again as a recollection of the results obtained in Section \ref{compm3}, in particular the Propositions~\ref{psca}, \ref{pscb}, \ref{pscc'}, \ref{pscd'} for the Poincar\'e polynomials of $\overline{\mathcal{A}}$, $\overline{\mathcal{B}}$, $\overline{\mathcal{C}}'$, $\overline{\mathcal{D}}'$ respectively. For the remaining twisted sectors, the Poincar\'e polynomial is computed applying Corollary~\ref{corollariocadman} and the results are summarized in Table~\ref{base0dim>0} and~\ref{base0dim0}. 

Finally, for the degree shifting numbers, one can compute directly the age of the twisted sectors using Proposition \ref{etamg}.
\end{proof}

\end{document}